\documentclass[11pt]{article}
 \usepackage[margin=0.5in]{geometry} 
\usepackage{amsmath,amsthm,amssymb,amsfonts}

\usepackage{graphicx}
\usepackage{verbatim}
\usepackage{float}
\usepackage[shortlabels]{enumitem}
\usepackage{placeins}
\usepackage{indentfirst}
\usepackage{multirow}
\usepackage{mathtools}
\usepackage{tcolorbox}
\usepackage{algorithm}
\usepackage{algpseudocode}
\usepackage{caption}
\usepackage{hyperref}
\usepackage{moreverb}
\usepackage{subcaption}
\newtheorem{proposition}{Proposition}
\newtheorem{remark}{Remark}

\begin{document}

\title{A Two-level GPU-Accelerated Incomplete LU Preconditioner for General Sparse Linear Systems\protect\thanks{This work was performed under the auspices of the U.S. Department of Energy by Lawrence Livermore National Laboratory under Contract DE-AC52-07NA27344. Work at LLNL was funded by Total S.A. through the FC-MAELSTROM Project.}}

\author{
Tianshi Xu \thanks{Department of Computer Science and Engineering, University of Minnesota, Minnesota, USA (\protect\href{mailto:xuxx1180@umn.edu}{xuxx1180@umn.edu})}
\and 
Ruipeng Li \thanks{Center for Applied  Scientific Computing, Lawrence Livermore National Laboratory, California, USA (\protect\href{mailto:li50@llnl.gov}{li50@llnl.gov})}
\and
Daniel Osei-Kuffuor \thanks{Center for Applied  Scientific Computing, Lawrence Livermore National Laboratory, California, USA (\protect\href{mailto:oseikuffuor1@llnl.gov}{oseikuffuor1@llnl.gov})}
}

\date{}

\maketitle

\abstract{This paper presents a parallel preconditioning approach based on incomplete LU (ILU) factorizations in the framework of  Domain Decomposition (DD) for general sparse linear systems. 
We focus on distributed memory parallel architectures, specifically, those that are equipped with graphic processing units (GPUs). 
In addition to block Jacobi, we present general purpose two-level ILU Schur complement-based approaches, where different strategies are presented to solve the coarse-level reduced system. 
These strategies are combined with modified ILU methods in the construction of the coarse-level operator, in order to effectively remove smooth errors. 
We leverage available GPU-based sparse matrix kernels to accelerate the setup and the solve phases of the proposed ILU preconditioner. 
We evaluate the efficiency of the proposed methods as a smoother for algebraic multigrid (AMG) and as a preconditioner for Krylov subspace methods, on challenging anisotropic diffusion problems and a collection of general sparse matrices.}

\textbf{keywords:}{ GPU computing; Preconditioning; ILU factorization; Multilevel methods; AMG }

\maketitle

\section{Introduction}\label{sec1}

This paper considers the problem of solving the linear system
\begin{equation}
    Ax=b,
    \label{eq:1-1}
\end{equation}
where $A\in \mathbb{R}^{n\times n}$ is large and sparse, which 
often arises in many fields of science and engineering to compute the numerical solutions of partial differential equations (PDEs).
Preconditioned Krylov subspace methods are one class of iterative solvers widely used for solving large sparse systems of linear equations. Here, much of the work has focused on the development of effective preconditioners that can solve the problem at hand in an efficient manner. 
A general rule of thumb of finding a good preconditioner, $M$, is to let  $M\approx A$ such that 
the eigenvalues of $M^{-1} A$ are clustered, which can lead to fast convergence of Krylov subspace methods. In addition, $M$ should be relatively inexpensive to compute and the application of $M^{-1}$ 
needs to be performed efficiently. 
As a result, the iterative solution of the preconditioned system 
can have better convergence and faster time-to-solution, compared to the
iterative solution of the original system \eqref{eq:1-1}.
We want to emphasize that when considering 
the efficiency of building and applying preconditioners, 
one needs to also take into account  the underlying
computing platforms. 
The massively parallel processing paradigm on modern many-core processors, like the GPUs considered in this work, have made some
traditional preconditioning methods inappropriate on such devices, see e.g., \cite{RliSaadGPU,chow2015fine}.

Algebraic multigrid (AMG) is yet another class of iterative solvers widely used in the linear solver community, both as a stand-alone solver and as a preconditioner. The optimal convergence and scalability of AMG make it an attractive choice for most applications. However, its optimality is restricted to PDEs with elliptic properties. 

Incomplete LU (ILU) factorization preconditioners are a class of solvers widely used as general purpose preconditioners for Krylov solvers. 
Compared with AMG, they are less likely to fail when solving indefinite and ill-conditioned problems or handling irregular meshes.
\cite{bank1999multilevel,botta1999matrix,li2003parms,dillon2018hierarchical,chow2015fine}.
Extensions and modifications to classical ILU preconditioners, including modified ILU and shifted ILU (with complex shifts)  \cite{magolu2000preconditioning,erlangga2006comparison,van2007spectral,erlangga2006novel,xi2017rational,liu2018solving}, have been proposed for solving indefinite linear systems that arise from applications of Helmholtz equations or interior eigenvalue problems. Such problems are quite difficult to solve by iterative methods since  the spectrum of $A$ includes the origin \cite{chow1997experimental}. 
Furthermore, ILU methods can also be used as a complex smoother within AMG, thus providing an alternative option for complex problems where standard relaxation-based approaches such as Jacobi and Gauss-Seidel exhibit slow convergence.

While ILU methods can be applied to a wider range of problems, the sequential nature of the algorithm limits its applicability to large scale applications on distributed systems. As a result, there exists active research in the development of efficient parallel ILU strategies. DD-based strategies are some of the most promising strategies. The simplest form of the DD approach is the block Jacobi ILU approach, where only the local block corresponding to individual subdomains are factorized, ignoring any inter-domain coupling. In \cite{luo2012hybrid}, the authors proposed a DD-based global Schwarz preconditioning, where the ILU is used as a complex smoother for local geometric multigrid, and GPU is used to accelerate the triangular solve. 
A more accurate alternative is the so-called two-level ILU strategy, where a local factorization is first performed within each subdomain, followed by a global strategy to account for the inter-domain coupling. 
Common approaches for the global strategy include a graph-based global ILU factorization of the coupling matrix (Schur complement) \cite{karypis1997parallel, hysom1999efficient}; and standard algebraic Schur complement strategies \cite{saad1999bilutm, cerdan2017two}. 
Partial ILU techniques or incomplete triangular solves may be used to control the sparsity of the Schur complement in the latter approach \cite{nievinski2018parallel}.
The work presented in this paper utilizes the standard algebraic Schur complement strategies with partial ILU techniques.
In \cite{cerdan2017two}, a two-level ILU preconditioner is developed specifically for electromagnetic applications. 
There, the global Schur complement is assembled on a single node and factorized sequentially. While the authors mention its suitability on distributed systems, the algorithm is evaluated on a single node using shared-memory parallelism. Perhaps, the most relevant works in the literature related to the approach proposed in this paper are \cite{saad1999bilutm} and \cite{nievinski2018parallel}. In \cite{saad1999bilutm}, Saad and Zhang proposed a parallel multilevel threshold-based ILU preconditioning technique based on DD. A partial ILU strategy is used to construct the Schur complement, which may be further reduced to multiple levels. The last level matrix is then assembled on a single node and factorized. In \cite{nievinski2018parallel}, the authors propose a two-level level-based ILU preconditioner that also utilizes partial ILU to construct the Schur complement corresponding to the global coupling. Here, they form a local Schur complement on each subdomain following \cite{carvalho2001algebraic}, and solve them locally to approximately solve the global Schur complement system. Both of these methods, while applicable to distributed systems, do not utilize potential enhancements offered by GPUs. It is worth mentioning that the notion of global Schur complement updates in the context of DD is also applicable to parallel sparse direct solvers \cite{sao2014distributed}.

Since the advent of CUDA, the NVIDIA GPUs have gained a lot of attention for accelerating sparse linear solvers \cite{doi:10.1137/140980260,RENNICH2016140,sao2014distributed,wang2009solving,CLARK20101517,6749152,Gandham20141151,doi:10.1080/17445760802337010,RliSaadGPU,doi:10.1137/15M1026419} and eigensolvers
\cite{Anzt:2015:ALM:2872599.2872609,dziekonski_rewienski_sypek_lamecki_mrozowski_2017,AURENTZ2017332,evslsisc}, 
the performance gains of which are usually from the accelerated sparse matrix computation kernels. 
GPUs have also been used to accelerate parallel ILU factorizations. Packages such as PARALUTION \cite{paralution} and HIFLOW \cite{emclpp42879} provide distributed memory ILU factorizations with GPU support. However, the implementation in these packages follow a block-Jacobi approach. 

{
In this paper, we present a DD-based two-level parallel ILU strategy designed for distributed memory systems and with GPU support. We summarize below the main contributions of this paper.
\begin{itemize}
    \item 
    We provide a detailed analysis of the two-level strategy, showing the benefit of using modified ILU to achieve faster convergence.
    
    \item 
    The algorithm supports several parallel ILU strategies, including the standard block Jacobi approach; a two-level additive ILU approach, where the coarse grid is obtained implicitly via a partial ILU factorization; and a two-level multiplicative ILU approach, where the coarse grid system is obtained via a Galerkin product akin to AMG methods.
    \item 
    This work is implemented within the hypre \cite{falgout2002hypre} package, making it readily available to application developers. 
    The implementation is MPI-based, which supports distributed memory systems, and the ILU preconditioners proposed in this paper take advantage of the existing accelerated kernels for computing the ILU(0) factorization and solving sparse triangular linear systems on GPUs.
\end{itemize}
}

The remainder of this paper is organized as follows: 
in Section~\ref{sec:parILUDD} we review the DD framework and present several parallel preconditioning algorithms of 
ILU factorizations based on DD. We then present the   
parallel implementation details of these algorithms with respect to GPUs in Section~\ref{sec:parDetails}. 
In Section~\ref{sec:results}, we present numerical results on our evaluation of the performance of the different preconditioning strategies, 
and we conclude in Section~\ref{sec:conclusion}.

%
%

\section{Parallel ILU via domain decomposition}\label{sec:parILUDD}

DD methods are 
widely used for solving coupled systems of PDEs over physical domains with different properties, particularly in the context of parallel computing.
The most general and versatile approach of DD, 
which is also purely algebraic, 
often resorts to partitioning the underlying graph of the matrix 
with graph partitioners 
\cite{Aykanat99,CHACO,METIS-SIAM,kolda98partitioning,SCOTCH,Simon-RSB}. 
 Here we assume that the global system has been partitioned 
into $p$ subdomains, each assigned a set of equations, which is stored locally.
Then, corresponding to this partitioning, the global system~\eqref{eq:1-1} can be written as
\begin{equation}\label{eq:equ}
\begin{pmatrix}
{A}_1 & E_{12} & \cdots & E_{1p} \\
E_{21} & {A}_2 & \cdots & E_{2p} \\
\vdots & \vdots & \ddots & \vdots \\
E_{p1} & E_{p2} & \cdots & {A}_p \\
\end{pmatrix}
\begin{pmatrix}
x_1 \\
x_2 \\
\vdots \\
x_p \\
\end{pmatrix}
=
\begin{pmatrix}
b_1 \\
b_2 \\
\vdots \\
b_p \\
\end{pmatrix},
\end{equation}
where
the local system for subdomain $i$ reads
\begin{equation} \label{eq:locsys}
A_i  x_i + \sum_{j \in \mathcal{N}_i} E_{ij} x_j = b_i,
\end{equation}
where  $A_i$ is the operator that acts on the local
degrees of freedom $x_i$,
$\mathcal{N}_i$  is a set  of the indices of the
subdomains that are adjacent  to  $i$, and 
$E_{ij}$ represents the connections between subdomains $i$ and $j$. 

\subsection{One-level block Jacobi approach}
One of the simplest ways to precondition \eqref{eq:equ} in parallel would be to use the block Jacobi approach. In this approach, 
 \eqref{eq:equ} is approximately solved by ignoring all the off-diagonal coupling in the coefficient matrix. For each subdomain, 
the local system $A_ix_i=b_i$ is solved instead, dropping
the coupling terms $E_{ij} x_j$ in \eqref{eq:locsys}.
Furthermore, the local systems are often approximately solved by
methods such as the ILU factorizations.
This approach is simple and requires no communication in 
building and applying the preconditioner. 
However, a common issue with this approach is that the convergence
of Krylov subspace methods combined with it generally deteriorates
as the number of subdomains increases and as the size of the system grows. In addition, ignoring off-diagonal terms limits the effectiveness of this block Jacobi strategy to systems with weak off-diagonal coupling. The convergence could be improved using restricted additive Schwarz method \cite{cai1999restricted}, but the memory cost is slightly higher. We note that although this block Jacobi strategy is not scalable (for the reasons mentioned above), it can be incorporated into AMG as an alternative to the standard smoothers when solving complex problems. 

\subsection{Two-level block ILU approach}

\subsubsection{Domain decomposition and local partitioning}
To reduce the sensitivity of the preconditioner to the number of subdomains, the two-level approach further takes advantage of the DD framework by first partitioning and reordering the nodes of the local subdomain into interior and exterior nodes.
Fig.~\ref{fig:DD} shows a DD of a 2-D grid
into 4 subdomains. 
The interior nodes in each subdomain connect only to other nodes inside the subdomain. These are identified by the (unfilled) circles in Fig.~\ref{fig:DD}. In contrast, the exterior nodes, black dots in Fig.~\ref{fig:DD}, can have connections to nodes from other subdomains. These connections are shown by the solid black lines. 
If we re-label the unknowns such that all the interior nodes are labeled first, followed by the exterior nodes,  
the global coefficient matrix can be reordered as shown in the right graph of Fig.~\ref{fig:DD}, where the upper left part of the matrix is now block diagonal, and each diagonal block corresponds to a subdomain. 

\begin{figure}[htb] 
\begin{center} 
\includegraphics[width=0.482\textwidth]{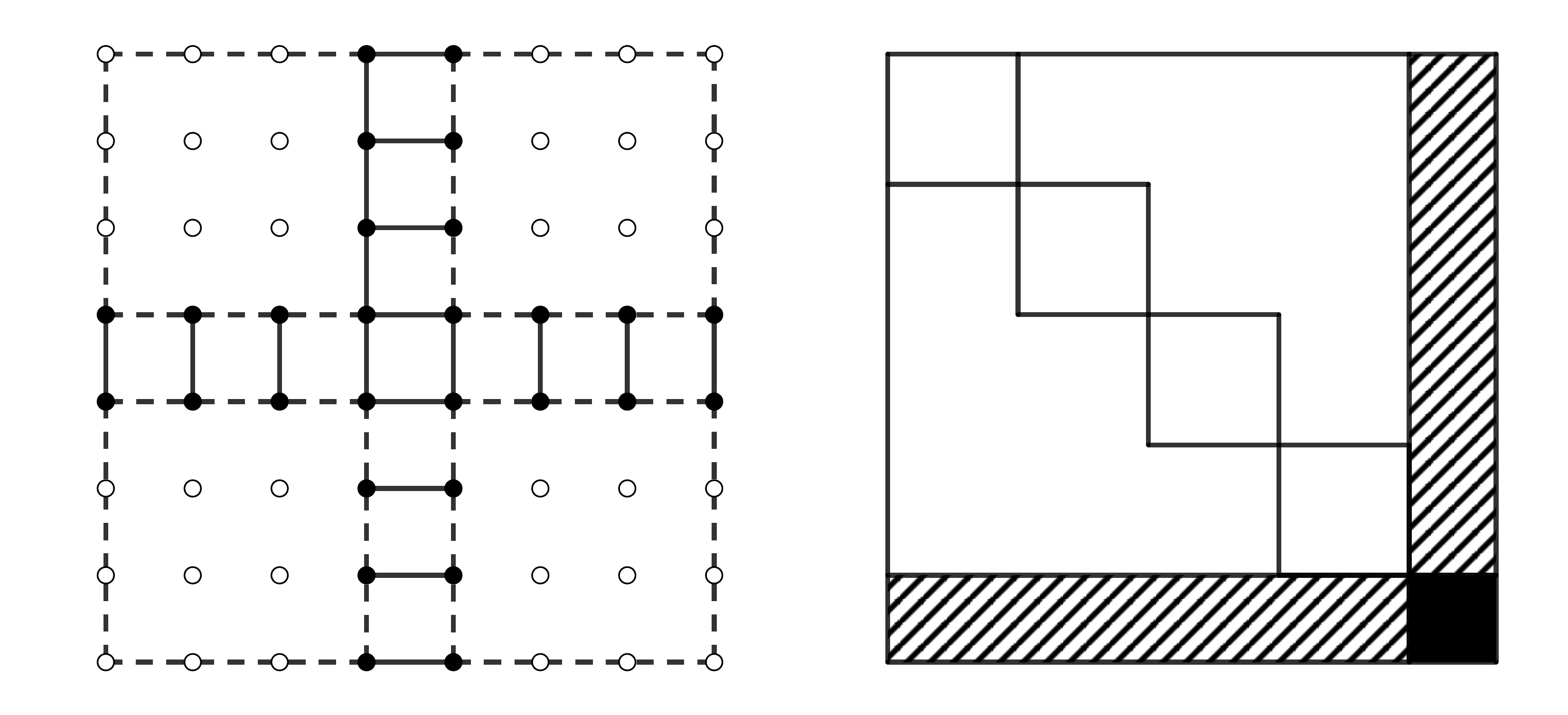}
\end{center} 
\caption{DD of mesh points (left) and the corresponding reordered global coefficient matrix (right).
\label{fig:DD}}
\end{figure}

From the point of view of the system of equations, the reordered system $Ax=b$, corresponding to the DD can be written as
\begin{equation} 
\begin{pmatrix}
 B  &   F \cr 
 E  &   C 
\end{pmatrix}
\begin{pmatrix}
 u \cr v
\end{pmatrix}
 = 
\begin{pmatrix}
 f \cr g
\end{pmatrix}
, 
 \label{eq:equ22} 
\end{equation}
where $B\in\mathbb{R}^{n_1\times n_1}$ is  block diagonal associated with the interior nodes, $C\in\mathbb{R}^{n_2\times n_2}$ corresponds to the exterior nodes, and $E$ and $F$ contain the local couplings.
We can write the block LDU factorization of $A$ of the form:
\begin{equation}
\begin{pmatrix}
 I  &    \cr 
 EB^{-1}  &   I 
\end{pmatrix}
\begin{pmatrix}
 B  &    \cr 
   &   S 
\end{pmatrix}
\begin{pmatrix}
 I  &   B^{-1} F \cr 
   &   I 
\end{pmatrix}
\begin{pmatrix}
 u \cr v
\end{pmatrix}
 = 
\begin{pmatrix}
  f \cr g
\end{pmatrix}
, 
\label{eq:blockldu} 
\end{equation}
where $S=C-EB^{-1} F$  is the global Schur complement matrix. 
It is easy to verify that the inverse of the block LDU factorization
of $A$ is given by
\begin{equation}
A^{-1}=
\begin{pmatrix}
 I  &   -B^{-1} F \cr 
   &   I
\end{pmatrix}
\begin{pmatrix}
 B^{-1}  &    \cr 
   &   S^{-1} 
\end{pmatrix}
\begin{pmatrix}
 I  &    \cr 
 -EB^{-1}  &   I 
\end{pmatrix}
.
\label{eq:blockldusolve} 
\end{equation}

\subsubsection{Reduced system with the global Schur complement} 
The typical approach of solving \eqref{eq:equ22}
is to first solve for the interface unknowns $v$ from the reduced system with the Schur complement, which is then substituted back into \eqref{eq:equ22} to obtain the solution for interior unknowns $u$.
Note here that, due to the block structure from DD, 
the solves with $B$ and the multiplications with $E$ and $F$ 
can be done locally without communicating with other processors. 
The only part that requires communication is in solution of the Schur complement system. 
Moreover, the global Schur complement matrix $S=C-EB^{-1} F$ can be constructed
efficiently in parallel. The local computations $S_i=C_i-E_i B_i^{-1} F_i$, yield the diagonal 
blocks of $S$, and the remaining entries of $S$ are the same as the entries in the off-diagonal blocks of the 
submatrix $C$ of the original coefficient matrix.
Thus, it is easy to see that this approach is amenable to parallel computing. 
By and large, having a smaller Schur complement usually leads to  better parallel performance.
\begin{algorithm}[t]
\caption{Additive block ILU solve}\label{alg:2levelschur}
\begin{algorithmic}[1]
  \State Compute $\hat{g}=g-E B^{\sim 1} f$
  \State Solve $v=S^{\sim 1} \hat{g}$ ~with~ $S=C-E B^{\sim 1} F$
  \State Solve $u=B^{\sim 1}(f-Fv)$
\end{algorithmic}
\end{algorithm}

To construct a preconditioner using this framework, 
$B$ is approximated by its ILU factorization, $B\approx LU$,
and the Schur system $Sy=\hat g$ is only solved approximately.
The two-level block ILU approach is summarized in Algorithm~\ref{alg:2levelschur},
where $B^{\sim 1}$ and $S^{\sim 1}$ denote that the actions of $B^{-1}$
and $S^{-1}$ are approximately applied.
It is easy to see that the parallel efficiency of the preconditioning technique in Algorithm ~\ref{alg:2levelschur} 
depends on the strategy used to solve the global Schur complement system. One strategy for performing this solve
 with $S$ is to use a small number of GMRES iterations preconditioned with the block Jacobi 
preconditioner that consists of
the local Schur complements $S_i$.
The inverse of the local Schur complements
can be applied using their ILU factorizations. Note that this factorization can be obtained by factorizing the local $B$ matrix.
Two different implementations of this approach for solving the global Schur complement system 
will be discussed in detail in Section~\ref{sec:implSchur}.

\subsubsection{Galerkin product and coarse-grid correction} 
Another way to view the aforementioned block ILU methods 
is from the perspective of coarse correction as in the AMG methods.
We can rewrite \eqref{eq:blockldusolve} as
\begin{equation}
A^{-1}=
\begin{pmatrix}
G  &   P^\star
\end{pmatrix}
\begin{pmatrix}
B^{-1}  &    \cr 
       &   S^{-1} 
\end{pmatrix}
\begin{pmatrix}
G^T    \cr 
     R^\star
\end{pmatrix}
,
\label{eq:blockldu2}
\end{equation}
where $G=\left(I,0\right)^T$, $P^\star$ and $R^\star$ are the so-called
ideal interpolation and restriction operators, given by 
\begin{equation}
P^\star=
\begin{pmatrix}
-B^{-1} F \cr 
       I 
\end{pmatrix}
, \quad
R^\star=
\begin{pmatrix}
-EB^{-1}  &   I 
\end{pmatrix}
\label{eq:idealpr},
\end{equation}
respectively.
It is easy to verify that the coarse-grid operator computed by
the Galerkin product with $P^\star$ and $R^\star$ is the Schur complement,
i.e.,
$R^\star AP^\star = S$, and the exact solution can be obtained by
\begin{equation}
    {x}=
    (GB^{-1} G^T+P^\star S^{-1} R^\star){b}
    \label{eq:blockldusolve2}.
\end{equation}
In the context of preconditioning, practical interpolation and the restriction operators are usually constructed as sparser approximations to the ideal ones,
which in this work are assumed to have the form of
\begin{equation} \label{eq:PR}
P =
\begin{pmatrix}
-B^{\sim 1} F \\ I
\end{pmatrix} \quad \mbox{and} \quad
R =
\begin{pmatrix}
-E B^{\sim 1} & I
\end{pmatrix}.
\end{equation}
Proposition~\ref{prop:Pstar} describes the energy-norm minimization property of $P^\star$ and its connection with a general $P$. In what follows, 
we denote the range of an operator $H$ as $\text{Ran}(H)$.
\begin{proposition} \label{prop:Pstar}
For an SPD matrix $A$, the ideal interpolation $P^\star$ is the $A$-orthogonal projection of any $P$ of the form 
$P^T = (W^T, I)$ onto 
subspace $\mbox{Range}(G)^{\perp_A}$, 
where $G=(I,0)^T$.
Moreover, $\Vert P\Vert_A \ge \Vert P^\star\Vert_A, \forall 
P$. 
\end{proposition}
\begin{proof}
Let $Q_A=GB^{-1} G^T A$, which is an $A$-orthogonal projector. 
It is easy to verify that $(I-Q_A) P=P^\star$. For the second part,
$\Vert Pv\Vert_A^2=\Vert P^\star v\Vert_A^2+\Vert{(I-Q_A)P v}\Vert_A^2 \ge \Vert{P^\star v}\Vert_A^2$,
so $\Vert{Pv}\Vert_A \ge \Vert{P^\star v}\Vert_A$ and the result follows.
\end{proof}
Similar property can be also shown for $(R^\star)^T$.
Given $P$ and $R$ in \eqref{eq:PR},
Algorithm~\ref{alg:2levelrap} presents this two-level preconditioning approach. 
\begin{algorithm}[t]
\caption{Two-level multiplicative solve}\label{alg:2levelrap}
\begin{algorithmic}[1]
  \State Compute $\hat x = GB^{\sim 1} G^T b$ \Comment{F-relaxation}
  \State Compute $r=R(b-{A}\hat x)$ \Comment{Restriction}
  \State Solve $v = S^{\sim 1} r$ ~with~ $S=R{A}P$\Comment{C-correction}
  \State Compute ${x}=\hat{x}+Pv$\Comment{Interpolation}
\end{algorithmic}
\end{algorithm}
\begin{remark}[Comparison between Algorithm~\ref{alg:2levelschur} and Algorithm~\ref{alg:2levelrap}]~ 

\begin{enumerate}
\item{The solution from Algorithm~\ref{alg:2levelschur} can be reformulated as \eqref{eq:blockldusolve2} with $P$ and $R$ defined in \eqref{eq:PR}, where the  relaxation $GB^{\sim 1} G^T$ and the coarse-grid correction $P S^{\sim 1} R$ are  \textit{additively} \cite{xui1989parallel}} applied to $b$.
\item {
The solution $x$ from Algorithm~\ref{alg:2levelrap} can be
expressed as
 $$x = \left(G B^{\sim 1} G^T + P S^{\sim 1} R (I-AGB^{\sim 1} G^T)\right)b,$$
 which is \textit{multiplicative \cite{vassilevski2014reducing}} and equivalent to the additive form
 \eqref{eq:blockldusolve2} when $R=R^\star$, since $R^{\star}AG=0$.}
 \item{ The Schur complement in Algorithm~\ref{alg:2levelschur} denoted by
 $S_{A_1}=C-EB^{-1} F$ is equal to $R_0AP$ with $R_0=(0,I)$.
Comparing with the Schur complement in Algorithms~\ref{alg:2levelrap}, denoted by $S_{A_2}= RAP$, we have  $S_{A_2}-S_{A_1}= E B^{\sim 1} B B^{\sim 1} F- E B^{\sim 1} F$.}
\item{ When $B^{\sim 1}$ is sufficiently accurate, the preconditioning qualities of Algorithms~\ref{alg:2levelschur} and \ref{alg:2levelrap} are generally similar, whereas for less accurate $B^{\sim 1}$, the multiplicative approach in Algorithm~\ref{alg:2levelrap} is often better. On the other hand,
$S_{A_2}$ is generally denser than $S_{A_1}$ and more expensive to build.}
\end{enumerate}
\end{remark}

In this work, we consider an ILU(0) factorization for $B^{\sim 1}$ in order to benefit from optimized kernels that 
efficiently perform the factorization as well as forward and backward solves on GPUs. Using ILU(0), we found that the approach
in Algorithm~\ref{alg:2levelrap}, in general, has much faster convergence and
better overall performance.

\subsection{Modified ILU factorizations for building $P$}\label{sec:milu}
It is clear that the interpolation matrix $P$ of the form in \eqref{eq:PR} solely depends on the choice of $B^{\sim 1}=U^{-1} L^{-1}$. In this section, we will show 
constructing $P$ from a modified ILU(0) factorization of $B$ yields a better interpolation operator, compared to using standard ILU(0). To show this, it is 
useful to examine the error propagation of Algorithm~\ref{alg:2levelrap}, which can be written as
\begin{equation}
\label{eq:errorprop0}
e_{i+1} := (I-P (RAP)^{\sim 1} RA)(I-GB^{\sim 1} G^T A)e_{i},
\end{equation}
where $e_i$ denotes the error in the solution at iteration $i$.
Suppose $(RAP)^{\sim 1}$ is exact, and then $I-P(RAP)^{-1} RA$ is an A-orthogonal projector that has the range of $P$, denoted by ${\textnormal{\mbox{Ran}}}(P)$, as its kernel.
Therefore, to  reduce the ``smooth'' errors
that cannot be effectively removed by the smoothing step $I-GB^{\sim 1} G^T A$ in \eqref{eq:errorprop0}, 
it is important to include the smooth errors in ${\textnormal{\mbox{Ran}}}(P)$.
For elliptic-type PDEs,  constant vectors represent the smoothest mode of $A$.
As a result,  standard AMG algorithms typically use  interpolation formulae
that can interpolate constant vectors exactly.

Using ILU, we show that the interpolation operator $P$ can also be constructed to
have a chosen vector in its range.
Specifically, in order to have a given vector
\begin{equation}
\label{eq:ranP1}
    \begin{pmatrix} y \\ z \end{pmatrix} \in {\textnormal{\mbox{Ran}}}(P), 
\quad \mbox{where} \quad P=\begin{pmatrix} -U^{-1} L^{-1} F \\ I \end{pmatrix},
\end{equation}
we need to let 
\begin{equation}
\label{eq:ranPconst}
-U^{-1} L^{-1} F z = y \quad \mbox{or} \quad LU y = -Fz.
\end{equation}
The ILU factorization of $B$ can be written as
\begin{equation}
    B = LU + H,
\end{equation}
where  $H$ contains the elements that are dropped 
entries
during the process of the factorization, so
\eqref{eq:ranPconst}
is equivalent to
\begin{align} \label{eq:Rconstraint}
(B-H) y = -Fz \quad \Leftrightarrow \quad 
H y = B y + F z \equiv w,
\end{align}
which can be satisfied by compensating the diagonal of $U$.
Consider the $i$-th step of a row-wise ILU factorization,  
\begin{equation}
    H_{i,:} = B_{i,:} - \sum_{j\le i} l_{ij}U_{j,:} \ , \quad l_{ii}=1,
    \label{eq:milutithrow}
\end{equation}
where $H_{i,:}$, $B_{i,:}$ and $U_{j,:}$ denote the $i$-th and the $j$-th row of $H$, $B$ and
$U$ respectively. In general, we do not explicitly have $H_{i,:}y = w_i$.  
However, it can be enforced by adding some perturbation to 
$u_{ii}$ to get  $u_{ii}+\Delta_{i}$. Correspondingly, $H_{i,:}$ becomes $H_{i,:}-\Delta_{i}e_i^T$, where $e_i$ is the $i$-th column 
of the identity matrix. Substituting this perturbed $H_{i,:}$ into \eqref{eq:Rconstraint}, we obtain
\begin{equation}
    \Delta_{i} = \frac{H_{i,:}y-w_i}{y_i}.
    \label{eq:milutshift}
\end{equation}
For elliptic PDE operators, the vectors of interest $y$ and $z$ are constant vectors, and furthermore, $By+Fz=0$. 
Thus, $\Delta_i$ in \eqref{eq:milutshift} reduces to
$\Delta_i=\sum_j H_{i,j}$, i.e., the sum of the dropped entries during the factorization of row $i$,
which is the standard modified ILU (MILU) algorithm\cite{manteuffel1980incomplete}.

In this work, the MILU(0) factorization is used to compute  the interpolation and restriction operators, while a different
factorization such as standard ILU(0) can be used in the smoothing step.

%
%

\section{Parallel implementation details}\label{sec:parDetails}

In this section, we discuss the implementation details of the
ILU-based preconditioners mentioned in the previous sections
on both CPUs and GPUs.

\subsection{Block Jacobi preconditioner with ILU} 
The implementation of the block Jacobi ILU is rather straightforward.
During the setup phase, we apply an ILU factorization to each local diagonal block $A_i$  after re-labeling and ordering the local subdomain into interior and exterior nodes, as  discussed previously.
Fill-reducing reordering may be used to improve the accuracy and stability of the ILU factorization. 
In this work, we use the RCM strategy \cite{gibbs1976algorithm} to reorder the local diagonal matrix $A_i$ prior to the ILU factorization. 

For the GPU implementation, we use the state-of-the-art 
ILU(0) routine from the Nvidia cuSPARSE library, 
which is based on the level-scheduling algorithm 
\cite{RliSaadGPU, naumov2015parallel}. 
An analysis routine is required before the
actual numerical factorization, which  
generates the level information in order to
exploit the parallelism in the factorization.
After that, the rows in the same level can be
 factorized at the same time to utilize 
the many-core architecture of GPUs. 
This level-scheduling algorithm normally 
works well for many applications, especially in the cases where 
the number of levels is not too large.
In the solve phase, we also take advantage of the 
sparse triangular solve routine from the cuSPARSE library. 
This routine also utilizes  level-scheduling, so that the unknowns in the same level can be solved simultaneously. 

One of the major differences between the CPU  and the GPU 
implementations of the proposed preconditioners 
is the requirement of explicit reordering of the unknowns to match their 
label as interior or exterior nodes. 
In the CPU implementation, an explicit reordering of the matrix can be avoided, and the 
factorization and subsequent triangular solves can be performed implicitly with a permutation array.
However, the cuSPARSE library does not support the use of  permutation arrays, so the matrix has to be explicitly permuted before calling the factorization  routine. As a result, a copy of the matrix with the new ordering has to be saved. 
However, since the computed factors are stored in the same memory of the input matrix (without the unit diagonal of the $L$-factor),
storing the extra permuted matrix actually does not require more memory  compared with the CPU implementation, where the $L$ and $U$ factors are stored in different matrices.
Furthermore, switching between the $L$  and the $U$ factors in the triangular solves can be done by using the matrix descriptor of cuSPARSE,
which removes
the need of storing the two triangular matrices separately.
The setup and the solve phases of the block Jacobi ILU 
preconditioner on GPUs are summarized in Algorithm~\ref{alg:bj},
where the \texttt{gather} and the \texttt{scatter}
functions from \texttt{thrust} are used  for 
permuting the input and the output vectors
of the solve.

\begin{algorithm}
\caption{Block Jacobi ILU on GPU}\label{alg:bj}
\begin{algorithmic}[1]
\Procedure{BJILU-Setup}{$A$}\Comment{Distributed CSR matrix}
\State Compute the RCM ordering $p_i$ of ${A}_i$
\State Reorder ${A}_i$ to ${A}^{p_i}_{i}$ with $p_i$
\State Call cuSPARSE to compute ILU(0) of ${A}^{p_i}_{i}$ 
\State Setup triangular solve 
\EndProcedure
\Procedure{$x=$ BJILU-Solve}{$b$}
   \State Call \texttt{thrust::gather} on $b_i$ with permutation $p_i$
   \State Call cuSPARSE to solve $L_{A_i} U_{A_i} x_i=b_i$
   \State Call \texttt{thrust::scatter} $x_i$ with permutation $p_i$
\EndProcedure
\end{algorithmic}
\end{algorithm}

\subsection{Computing the global Schur complement and two-level additive solve} 
\label{sec:implSchur}
In this section, we discuss the details of the
two-level  ILU preconditioner based on Algorithm~\ref{alg:2levelschur}. 
Recall that the Schur complement from the DD is defined as $S=C-EB^{-1} F$.
Let $A_i$ denote the submatrix defined on subdomain $i$, from the DD. 
Suppose that the unknowns are locally ordered as interior and exterior nodes as 
previously discussed, then $A_i$ can be written as:
Assume that the local submatrix $A_i$ (corresponding to subdomain $i$) takes the form
\begin{equation}\label{eq:localA}
A_i = 
\begin{pmatrix}
B_i & F_i \\
E_i & C_i 
\end{pmatrix}
.
\end{equation}
An important result from the DD in ~\eqref{eq:equ22} is that
$B$ is block diagonal with $B_i$ on the diagonal. The same is true for $E$ with $E_i$ and $F$ with $F_i$.
Thus, the term $EB^{-1} F$ from $S$ is also block diagonal. Furthermore, 
the diagonal blocks of $S$, denoted by $S_i$, can be computed 
locally as $S_i=E_iB_i^{-1} F_i$.
The off-diagonal blocks of $S$ are the same as in $C$, denoted by $E_{ij}$, which couples the domains $i$ and $j$.
Therefore, the reduced Schur complement system $Sy=g'$ 
takes the following form,
\begin{equation}\label{eq:SS}
\begingroup
\setlength\arraycolsep{3pt}
\begin{pmatrix}
S_1 & {E}_{12} & \cdots & {E}_{1p} \\
{E}_{21} & S_2 & \cdots & {E}_{2p} \\
\vdots & \vdots & \ddots & \vdots \\
{E}_{p1} & {E}_{p2} & \cdots & S_p \\
\end{pmatrix}
\endgroup
\begin{pmatrix}
y_1 \\
y_2 \\
\vdots \\
y_p 
\end{pmatrix}
= 
\begin{pmatrix}
g_1' \\
g_2' \\
\vdots \\
g_p'
\end{pmatrix}
.
\end{equation}
In practice, computing the exact $S_i$ is usually 
too expensive and an approximation is preferred instead.
Approximations to $S_i$, denoted by $\tilde{S}_i$, 
 obtained  by dropping small
entries, can  be used in the place of $S_i$.
Forming the local matrices $\tilde S_i$ and thus having 
the explicit form of a matrix that approximates
the global Schur complement provides the flexibility of choosing different methods to solve the reduced system~\eqref{eq:SS}. 
Two approaches  for 
computing $\tilde{S}_i$ have been considered.
In the first approach we compute 
$\tilde S_i = C_i - E_i U_{B_i}^{-1}L_{B_i}^{-1}F_i$ with
the ILU factorization of $B_i$, where $L_{B_i}$ and $U_{B_i}$ denote the computed $L$ and $U$ factors of $B_i$. 
For the term $E_i U_{B_i}^{-1}L_{B_i}^{-1}F_i$, we first compute the two intermediate 
matrices $W_i = E_i U_{B_i}^{-1}$ and $Z_i = L_{B_i}^{-1}F_i$, and then compute $S_i = C_i-W_i Z_i$. 
Note that in the above computations, the 
sparse triangular solves to compute $W_i$ and $Z_i$  should be done carefully to
exploit the sparsity. Furthermore,  
 sparse matrix-matrix multiplications is needed for 
$W_i Z_i$. In practice it is often necessary to drop small entries  in order to keep the cost of these computations 
inexpensive and the resulting approximation $\tilde S_i$ sparse.

An alternative approach to compute $\tilde S_i$ that is more flexible and potentially more efficient, is to use the idea of partial ILU factorization. 
In this approach, the ILU factorization of $A_i$ skips the elimination steps within the $(2,2)$ block corresponding to the Schur complement. 
To be succint, the ILU factorization of  $A_i$ can be written as
\begin{equation}\label{eq:BILU}
\begin{pmatrix}
L_{B_i} &  \\
E_iU_{B_i}^{-1} & L_{S_i}
\end{pmatrix}
\begin{pmatrix}
U_{B_i} & L_{B_i}^{-1}F_i \\
 & U_{S_i} 
\end{pmatrix}
,
\end{equation}
whereas, the partial ILU factorization,  
can be written as
\begin{equation}\label{eq:pILU}
\begin{pmatrix}
L_{B_i} &  \\
E_iU_{B_i}^{-1} & I 
\end{pmatrix}
\begin{pmatrix}
U_{B_i} & L_{B_i}^{-1}F_i \\
 & {S_i} 
\end{pmatrix}
.
\end{equation}
Note that the difference between the two factorizations
is in the $(2,2)$ block, where the desired Schur complement
appears in the corresponding block of the $U$ matrix.
Figure~\ref{fig:pILU} provides an illustration of the partial ILU and the full ILU factorizations.
\begin{figure}[htb] 
\begin{center} 
\includegraphics[width=0.482\textwidth]{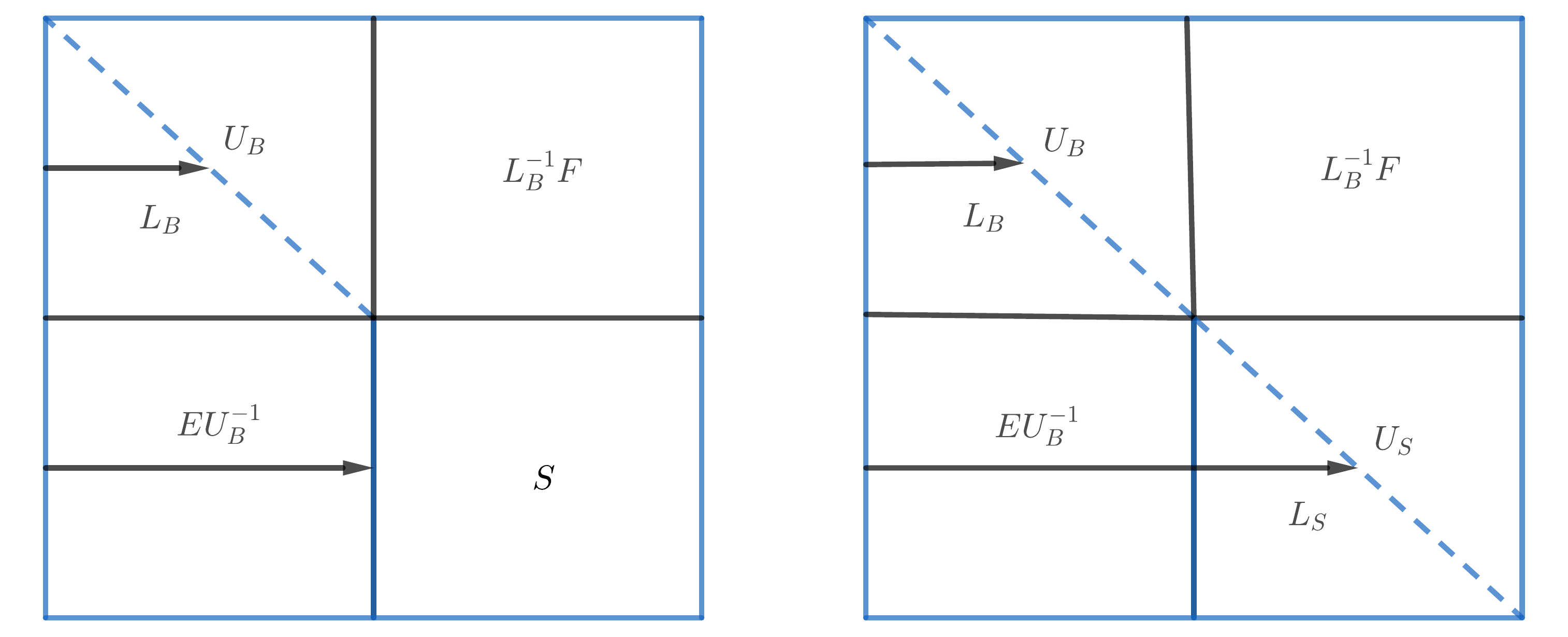}
\end{center} 
\caption{Illustration of the partial ILU factorization, 
which leaves an approximation to the Schur complement at the (2,2) block (left), and the normal ILU factorization, which computes 
an ILU factorization of the Schur complement as well (right)}
\label{fig:pILU}
\end{figure}

Compared to the previous approach for computing 
$\tilde S_i$, 
there exist several advantages on the  flexibility and the efficiency in the  partial ILU factorization approach:
\begin{enumerate}
\item Different dropping criteria can be applied in the 
different blocks of $A_i$, which provides 
full flexibility of controlling the
accuracy and the sparsity of the resulting matrices
$L_{B_i}$, $U_{B_i}$,
$E_iU_{B_i}^{-1}$, $L_{B_i}^{-1}F_i$ and $\tilde S_i$.
\item Matrices $W_i=E_iU_{B_i}^{-1}$, $Z_i=L_{B_i}^{-1}F_i$ and $\tilde S_i$
can be implicitly computed in a partial ILU factorization all at once,
which is typically more efficient than 
computing them separately as in 
the pervious approach.
\end{enumerate}
In the CPU implementation, we adopt 
the latter approach of using
the partial ILU factorization to compute $\tilde S_i$.
In Algorithm \ref{alg:2levelschur}, two solves with $B_i$ are required by applying $U_B^{-1} L_B^{-1}$.
On the other hand, if the intermediate matrices that
approximate $W=EU_B^{-1}$ and $Z=L_B^{-1} F$ from the partial ILU factorization are stored as well, then the two-level solve can be performed with only one solve with $L_{B_i}$ and $U_{B_i}$, as follows:

\begin{algorithm}
\begin{algorithmic}[1]
   \State Local solve $f_i'=L_{B_i}^{-1}f_i$
   \State Compute $g_i'=g_i-W_if_i'$
   \State Global solve $Sy=g'$
   \State Local solve $u_i=U_{B_i}^{-1}(f_i'-Z_iy_i)$
\end{algorithmic}
\end{algorithm}

For the GPU implementation, 
there are several technical difficulties for using the afore mentioned strategies to compute $\tilde S_i$ explicitly.
First, there is no available routine to compute the partial ILU(0)
factorization on GPUs; Second, efficient implementation of sparse triangular solves with 
multiple sparse right-hand sides and 
sparse matrix-matrix multiplications, needed to compute $\tilde S_i = C_i - \left(E_i U_{B_i}^{-1}\right)\left(L_{B_i}^{-1}F_i\right)$,  can  be challenging on GPUs.
The strategy adopted for the GPU implementation follows from reformulating the  system in \eqref{eq:SS} 
by preconditioning with a diagonal (block Jacobi) preconditioner. The resulting linear system has the following form:
\begin{equation}\label{eq:preSS}
\begingroup 
\setlength\arraycolsep{3pt}
\begin{pmatrix}
I & \tilde S_1^{-1}E_{12} & \cdots & \tilde S_1^{-1}E_{1p} \\
\tilde S_2^{-1}E_{21} & I & \cdots & \tilde S_2^{-1}E_{2p} \\
\vdots & \vdots & \ddots & \vdots \\
\tilde S_p^{-1}E_{p1} & \tilde S_p^{-1}E_{p2} & \cdots & I \\
\end{pmatrix}
\endgroup
\begin{pmatrix}
y_1 \\
y_2 \\
\vdots \\
y_p 
\end{pmatrix}
= 
\begin{pmatrix}
\tilde S_1^{-1}g_1' \\
\tilde S_2^{-1}g_2' \\
\vdots \\
\tilde S_p^{-1}g_p'
\end{pmatrix}
.
\end{equation}
Notice that the explicit form of $\tilde{S}_i$ is not needed and only the application of  $\tilde{S}_i^{-1}$ is required. 
$\tilde{S}_i^{-1}$ is approximated by $U_{S_i}^{-1} L_{S_i}^{-1}$, where $L_{S_i}$ and $U_{S_i}$ 
are the submatrices corresponding to the Schur complement block of the ILU(0) factors of ${A}_i$. We remark that
$L_{S_i}U_{S_i}$ is not the ILU(0) factorization of the Schur complement, since the sparsity patterns of $L_{S_i}$ and $U_{S_i}$ match that of $C_i$.
In our implementation, we  compute the ILU(0) factorization
of ${A}_i$ and then extract the factors of $L_{B_i}$, $U_{B_i}$, $L_{S_i}$ and $U_{S_i}$.
Krylov subspace iterations are applied to \eqref{eq:preSS},
which only requires the coefficient matrix in the form of MATVEC. The solution is achieved by performing: 
\begin{equation}\label{eq:preSMV2}
y_i + \tilde S_i^{-1} \sum_{j \in \mathcal{N}_i} \hat E_{ij} y_j,
\end{equation}

within each subdomain, where a MATVEC
with the  off-diagonal matrix of $A$ is needed,
followed by  triangular solves with $L_{S_i}$ and $U_{S_i}$. 
Therefore, the communication cost of the MATVEC with the coefficient matrix in \eqref{eq:preSS} is exactly
the same as that with the original matrix $A$.
The  implementation of the two-level ILU solve on GPUs is presented in Algorithm~\ref{alg:gmres}.


\begin{algorithm}[t]
\caption{Additive SCHUR-ILU on GPU}\label{alg:gmres}
\begin{algorithmic}[1]
\Procedure{SCHUR-ILU Setup} {$A$} \Comment{Distributed CSR matrix}
   \State Compute local permutation $p_i$
   \State Permute ${A}_i$ to ${A}^{p_i}$ with  $p_i$
   \State Call cuSPARSE to compute ILU(0) of ${A}^{p_i}$
   \State Extract factors $L_{B_i}$, $U_{B_i}$
   $L_{S_i}$, $U_{S_i}$,
   $E_iU_{B_i}^{-1}$, and $L_{B_i}^{-1}F_i$ 
   \State Setup triangular solve
   \State \textbf{return} $L_{B_i}$, $U_{B_i}$
   $L_{S_i}$, $U_{S_i}$,
   $E_iU_{B_i}^{-1}$, and $L_{B_i}^{-1}F_i$
\EndProcedure
\Procedure{SCHUR-ILU Solve}{$b$}\Comment{Solve for x with right-hand-side b}
   \State Call \texttt{thrust::gather} to permute $b_i$ with  $p_i$
   \State Call cuSPARSE to solve for $f_{i}' = L_{B_i}^{-1}f_{i}$
   \State Compute $g_i' = g_{i} - E_iU_{B_i}^{-1} f_{i}'$
   \State Apply GMRES to  \eqref{eq:preSS} to solve for $y$
   \State Compute $f_i'' = f_{i}' - L_{B_i}^{-1}F_i y_{i}$
   \State Call cuSPARSE to solve $u_{i}=U_{B_i}^{-1}f_i''$
   \State Call \texttt{thrust::scatter} to permute $x_i$ with  $p_i$
   \State \textbf{return} $x$
\EndProcedure
\end{algorithmic}
\end{algorithm}

\subsection{Two-level multiplicative solve}

In this section, we discuss the implementation of  the two-level multiplicative approach corresponding to Algorithm~\ref{alg:2levelrap}.  
MILU(0) factorization was used to construct $R$ and $P$, while  standard ILU(0) is used as the smoother. Thus, two different ILU(0) 
factorizations were required, which doubles the memory cost. Nonetheless, this combination often yielded better overall convergence  
compared to using MILU(0) as the smoother. 
From \eqref{eq:PR}  and subsequent discussions herein, it is easy to see that computing the  exact $R$ and $P$ and the subsequent 
$R{A}P$ product is impractical.
Moreover, explicitly forming these matrices is 
particularly difficult on GPUs as mentioned before.
Therefore, in this approach, $R$, $P$ and $RAP$ are unformed. Instead, 
 only $L_B$ and $U_B$ are stored, since the matrix $RAP$ is needed only in the form of the MATVEC to apply Krylov subspace methods on reduced Schur complement system. 
With this unformed $RAP$, the MATVEC 
requires four local
triangular solves with $L_{B_i}$ and $U_{B_i}$, two for applying $P$
and two for multiplying $R$ respectively.
One way to reduce the cost of these evaluations is to use approximations of 
$W\equiv EU_B^{-1}$ and $Z\equiv L_B^{-1} F$ that is available
from the ILU factorization of $A$, which can reduce the number of the triangular solves required in the 
MATVEC with $RAP$ by a factor of two.
Assume that the ILU factors of ${A}$ are
\begin{equation}
L = 
\begin{pmatrix}
L_B & \cr
\tilde W & L_S
\end{pmatrix}
\quad \mbox{and} \quad
U
=
\begin{pmatrix}
U_B & \tilde Z \cr
 & U_S
\end{pmatrix}
,
\label{eq:ilua}
\end{equation}
where $\tilde W$ and $\tilde Z$  approximate  $W$ and $Z$ respectively. Then, it follows that  
 the $P$ and $R$ in \eqref{eq:PR} can be
given by
\begin{equation}
P = 
\begin{pmatrix}
 -U_B^{-1} \tilde Z \cr
  I   
\end{pmatrix}
\quad \mbox{and} \quad
R
=
\begin{pmatrix}
-\tilde WL_B^{-1} & I
\end{pmatrix}
,
\label{eq:wzinmilu}
\end{equation}
where $-U_B^{-1} \tilde Z$ and $-\tilde W L_B^{-1}$ are kept unformed.  As a result, only 
 two solves with $L_B$ and $U_B$ are needed to perform the MATVEC with $RAP$.
In Section~\ref{sec:milu}, we showed how  MILU(0) of $B$ can be used to construct
$P$ in \eqref{eq:ranP1}. Similarly, MILU(0) of $A$ can also be utilized here to
ensure that a given vector  is in $\textnormal{\mbox{Ran}}(P)$. To see this, 
suppose the ILU factorization ${A}=LU+H$ has the following 2-by-2
block form,
\begin{equation}
\begin{pmatrix}
B & F \\
E & C
\end{pmatrix} =
\begin{pmatrix}
L_BU_B+H_{11} & L_B \tilde Z+H_{12} \cr
\tilde W U_B+H_{21} & \tilde W \tilde Z+L_SU_S+H_{22}  
\end{pmatrix}.
\label{eq:iluares}
\end{equation}
Having  vector $\begin{pmatrix}
y \\ z
\end{pmatrix} \in \textnormal{\mbox{Ran}}(P)$ leads to the following equation,
\begin{equation}
Pz = 
\begin{pmatrix}
-U_B^{-1} \tilde Z \cr
 I
\end{pmatrix}
z
=
\begin{pmatrix}
 y \cr
 z   
\end{pmatrix}
,
\label{eq:resinmilu}
\end{equation}
and thus
\begin{equation}
U_By = -\tilde Zz \; \Rightarrow \; L_BU_By = - L_B \tilde Z z,
\label{eq:milutwz1}
\end{equation}
together with \eqref{eq:iluares}, we have
\begin{equation}
\left(B-H_{11}\right)y = \left(H_{12}-F\right)z,
\end{equation}
which can be rewritten as
\begin{equation}
H_{11}y + H_{12}z = B y + Fz
     \equiv
     w
    \label{eq:miluwz}
\end{equation}
which leads to the MILU of $A$.

From \eqref{eq:resinmilu}, the corresponding $RAP$ is now
\begin{equation}
RAP =
\begin{pmatrix}
-\tilde W L_B^{-1} & I
\end{pmatrix}
\begin{pmatrix}
B & F \cr
E & C
\end{pmatrix}
\begin{pmatrix}
-U_B^{-1} \tilde Z \cr
I    
\end{pmatrix}
.
\label{eq:miluwz2}
\end{equation}
The remaining task is to build a preconditioner for the coarse-grid operator $RAP$,
so that we can effectively apply Krylov subspace methods to solve the 
coarse-grid problem.
Expanding the expression of $RAP$ in \eqref{eq:miluwz2} and using \eqref{eq:iluares}, $RAP$ can be expressed as
\begin{equation}
RAP = L_SU_S+RHP.
\label{eq:miluwz3}
\end{equation}
The residual matrix $H$ contains  the dropped elements during
the ILU factorization. These elements decrease in size as the 
 accuracy of the factorization increases, 
and when the factorization is exact, $H$ is zero.
Thus, it makes sense to use $L_SU_S$ as the preconditioner of $RAP$.
Moreover, if $L_S$ and $U_S$ are
obtained from the MILU(0) factorization of $A$ as described above,
the preconditioner can also preserve the action on the constant vector. That is,
\begin{equation}
RAP \mathbf{1} = L_s U_s \mathbf{1},     
\end{equation}
where $\mathbf{1}$ denotes the vector of all ones,
which can be easily shown from $P\mathbf{1}=\mathbf{1}$ and $H\mathbf{1}=0$.
This property is known to be important for preconditioning elliptic operators~\cite{manteuffel1980incomplete}. 

In our implementation, we compute the MILU(0) factorization of each $A_i$, and extract the factorization of $B_i$ from the (1,1) blocks of the factors, while 
the $\tilde W_i$ and $\tilde Z_i$ matrices are available in the corresponding
(2,1) and (1,2) blocks. The remaining (2,2) blocks, which are the ILU factorization of $S_i$, is used as the preconditioner for $S_i$.
A technical problem for the GPU implementation 
of this approach
is that there is currently no 
available routine for MILU(0) on GPUs. For this reason,
we simply leave the computation of the MILU(0) factorization on CPU for now. 
We remark that adding the MILU(0) option to the existing ILU(0) implementation
on GPUs
is straightforward, which is left as our future work.
The implementation of the GPU-based two-level multiplicative ILU method corresponding to Algorithm~\ref{alg:2levelrap} is presented in Algorithm~\ref{alg:rap}.
\begin{algorithm}
\caption{Multiplicative RAP-ILU on GPU}\label{alg:rap}
\begin{algorithmic}[1]
\Procedure{RAP-ILU Setup} {$A$} \Comment{Distributed CSR matrix}
   \State Compute local permutation $p_i$ 
   \State Permute ${A}_i$ to ${A}^{p_i}$ with  $p_i$
   \State Call cuSPARSE to compute ILU(0) $L_{A_i}U_{A_i}\approx{A}^{p_i}$
   \State Compute MILU(0) $\tilde L_{A_i}\tilde U_{A_i}\approx{A}^{p_i}$
   \State Extract factors $L_{B_i}$, $U_{B_i}$
   $L_{S_i}$, $U_{S_i}$,
   $E_iU_{B_i}^{-1}$, and $L_{B_i}^{-1}F_i$ 
   \State Setup triangular solve
   \State \textbf{return} $L_{A_i}$, $U_{A_i}$, $L_{B_i}$, $U_{B_i}$
   $L_{S_i}$, $U_{S_i}$,
   $E_iU_{B_i}^{-1}$, and $L_{B_i}^{-1}F_i$
\EndProcedure
\Procedure{RAP-ILU Solve}{$b$}\Comment{Solve for x with right-hand-side b}
   \State Call \texttt{thrust::gather} to permute $b_i$ with  $p_i$
   \State Call cuSPARSE to solve for $L_{A_i}U_{A_i}x_{i} = b_{i}$
   \State Call cuSPARSE to compute $r = R(b-\hat{A}x)$ 
   \State Apply GMRES to solve $Sv=r$
   \State Call cuSPARSE to compute $\hat{x}_{i}=x_{i} + P_ir_i$
   \State Call \texttt{thrust::scatter} to permute $x_i$ with  $p_i$
   \State \textbf{return} $x$
\EndProcedure
\end{algorithmic}
\end{algorithm}

%
%

\section{Numerical Experiments}\label{sec:results} 

The numerical experiments for evaluating the proposed preconditioners were performed on the 
HPC clusters Ray and Lassen, both at the Lawrence Livermore National Laboratory.  
{Each node of Ray has 4 NVIDIA P100 GPUs and 2 IBM POWER8 CPUs (dual-socket) with 10 cores. }
{Each node on Lassen has 4 NVIDIA V100 GPUs and 2 IBM POWER9 CPUs (dual-socket) with 22 cores.}
The CUDA program was compiled using \texttt{nvcc}  with the option 
\texttt{-gencode arch=compute\_60,"code=sm\_60"} 
for P100, and \texttt{compute\_70, sm\_70} for V100 respectively.
The parallel ILU strategies presented in this paper have been implemented as part of a suite of parallel ILU smoothers and preconditioners, within the hypre \cite{falgout2002hypre} linear solver library. In what follows, we evaluate the performance of the various preconditioning techniques presented in this paper on elliptic PDE  model problems.
Here, we denote the block Jacobi ILU preconditioner by BJ-ILU, the additive two-level ILU preconditioner by SCHUR-ILU; the two-level multiplicative ILU preconditioner by RAP-ILU; and the two-level multiplicative MILU preconditioner by RAP-MILU.

\subsection{Convergence Result}

\subsubsection{ILU as preconditioners for GMRES}

We begin our experiment by evaluating the performance of different ILU methods as the preconditioners for flexible GMRES (FGMRES).
 The reason for using FGMRES is that the preconditioner is not fixed, since GMRES is used to solve the global Schur complement system.
Here, the GPU implementations of BJ-ILU, SCHUR-ILU, RAP-ILU, and RAP-MILU are evaluated. 
The restart dimension for 
FMGRES is set to $50$ (FGMRES(50)), with a relative convergence tolerance of $1.0e-8$.
We use the (M)ILU(0) variants of the 
factorization. Here, and in the remaining experiments, we use three steps of GMRES to 
solve the reduced system for the two-level methods.

The model problem is a standard Laplacian problem defined as:
\begin{equation} \label{eq:lap}
-\Delta u = f \quad \mbox{in} \quad \Omega=[0,1]^2 \; \mbox{or} \; [0,1]^3    
\end{equation}
with the 5-pt stencil for 2D problem, and 7-pt stencil for 3D problem. 

We run the first test using 16 nodes on Ray, giving a total of $64$ MPI processes, with 4 MPI processes per node. 
For the 2D problem, we test a problem size of $1024^2$, and for the 3D problem, we test a size of $512^3$.

In the following tables, $its$ is the number of outer iterations needed for the solver to converge, $p$-$t$ is the time of the setup phase, and $s$-$t$ is the time of the solve phase.

\begin{figure}[tbhp]
\centering
    \begin{subfigure}[tbhp]{.4\textwidth}
        \centering
        \includegraphics[width=\textwidth]{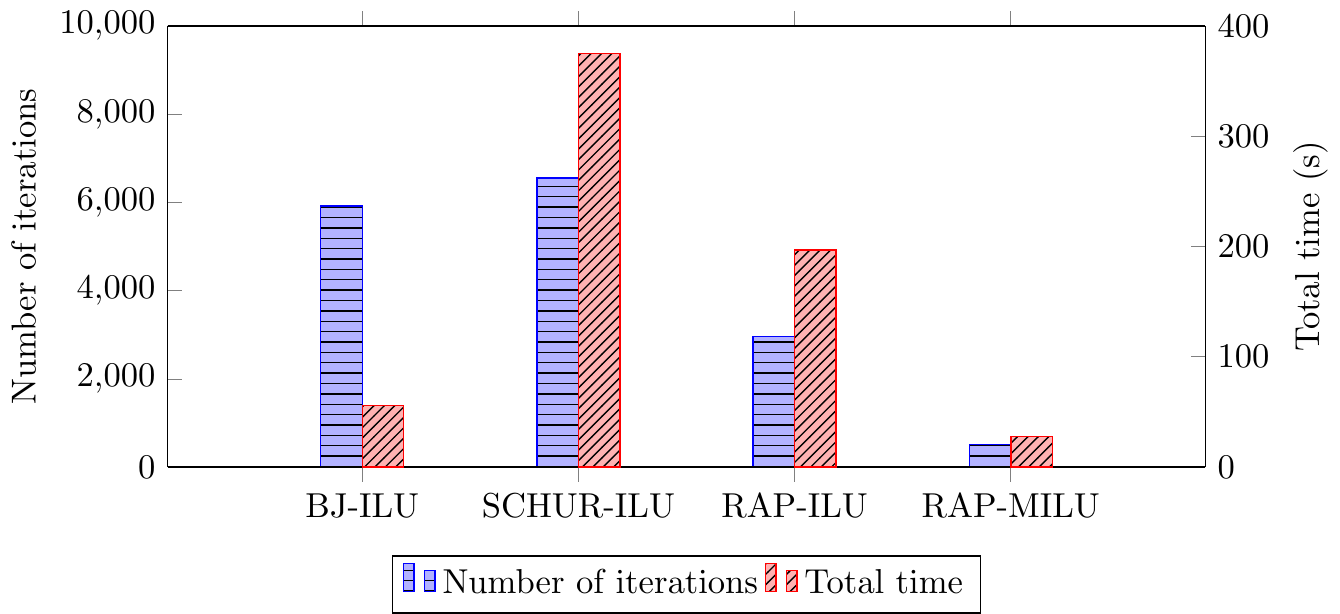} \\[0.5em]
        \setlength{\tabcolsep}{3pt}
        \begin{tabular}{lrrr} 
          \bf Preconditioner & \bf \#its & \bf setup & \bf solve  \\ \hline
          BJ-ILU(0) & 5919 & 0.45 & 55.42 \\ 
          SCHUR-ILU(0) & 6553 & 0.45 & 374.6 \\ 
          RAP-ILU(0) & 2962 & 0.47 & 169.32 \\ 
          RAP-MILU(0) & 514 & 0.51 & 27.12 \\ \hline
        \end{tabular}
        \caption{2D Poisson problem of size $1024^2$}
    \end{subfigure}
    \begin{subfigure}[tbhp]{.4\textwidth}
        \centering
        \includegraphics[width=\textwidth]{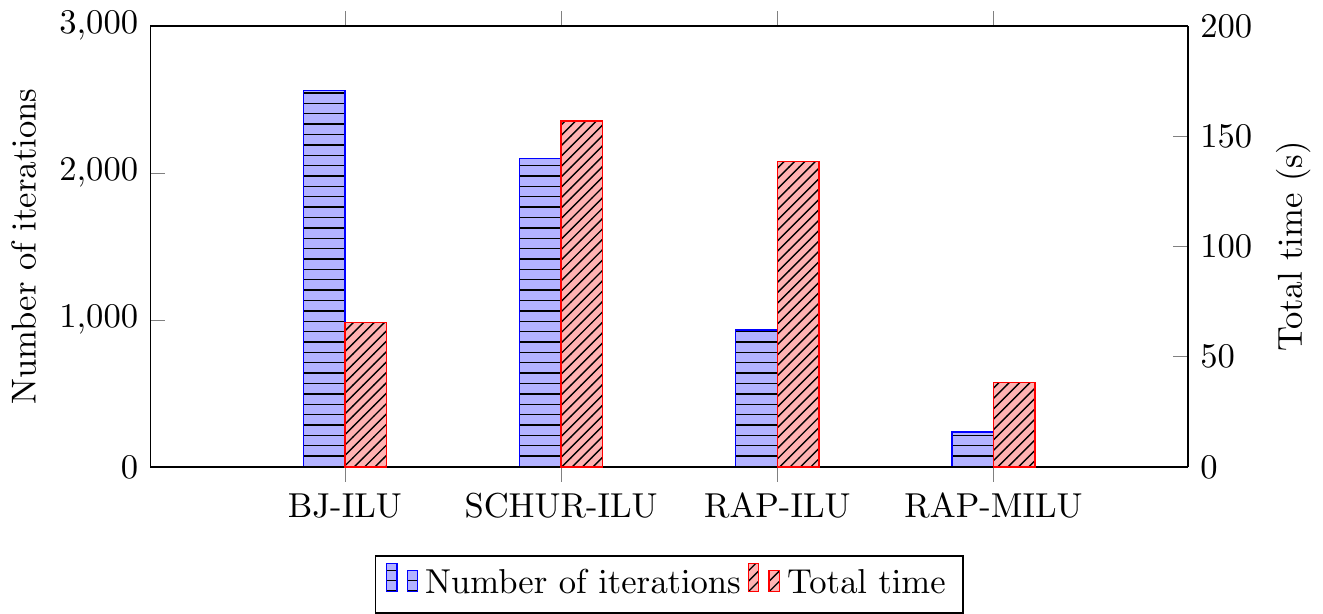} \\[0.5em]
                \setlength{\tabcolsep}{3pt}
          \begin{tabular}{lrrr} 
          \bf Preconditioner & \bf \#its & \bf setup  & \bf solve  \\ \hline 
           BJ-ILU(0) & 2561 & 1.23 & 64.38 \\ 
           SCHUR-ILU(0) & 2100 & 1.4 & 155.56 \\ 
           RAP-ILU(0) & 923 & 1.83 & 136.81 \\ 
           RAP-MILU(0) & 238 & 3.91 & 34.38 \\ \hline
          \end{tabular}
        \caption{3D Poisson problem of size $512^3$}
    \end{subfigure}
    \caption{Iteration counts and timings of the block Jacobi preconditioner with ILU(0) and the two-level ILU(0)/MILU(0) preconditioners for 2-D/3-D Poisson problems along with
    FGMRES(50). The runs used 64 processes on 16 nodes.}
    \label{fig:test_gmres_all}
\end{figure}

Figure~\ref{fig:test_gmres_all} highlights the results of this evaluation. The results indicate that the different methods are quite competitive. 
While BJ-ILU typically showed a slow rate of convergence, the time-to-solution was quite fast. 
On the other hand, SCHUR-ILU exhibits a slow rate of convergence that was proportational to its time-to-solution. 
The RAP strategies appear to be the better in both convergence rate and the time-to-solution for these test cases.
In particular, we can clearly see the benefit of using 
the modified ILU approach to construct the interpolation and restriction operators, thereby capturing important elliptic properties of the operator. Even though the setup
 costs for RAP-MILU is typically larger than the other options, the convergence rate was very fast, leading to the fastest total time.

In the next experiment, we perform a weak scaling study of those different strategies.
The experiments are performed on up to 16 nodes on Ray 
with up to 4 MPI processes per node, to solve the system in~\eqref{eq:lap}. For the 2D problem, 
we keep the problem size on each subdomain fixed at $256^2$. 
We only show the result up to $32$ processes, since BJ-ILU failed to converge within 20,000 steps. For the 3D problem, we keep the problem size in each subdomain as $128^3$.

\begin{figure}[tbhp]
\centering
    \begin{subfigure}[t]{.4\textwidth}
        \centering
        \includegraphics[width=\textwidth]{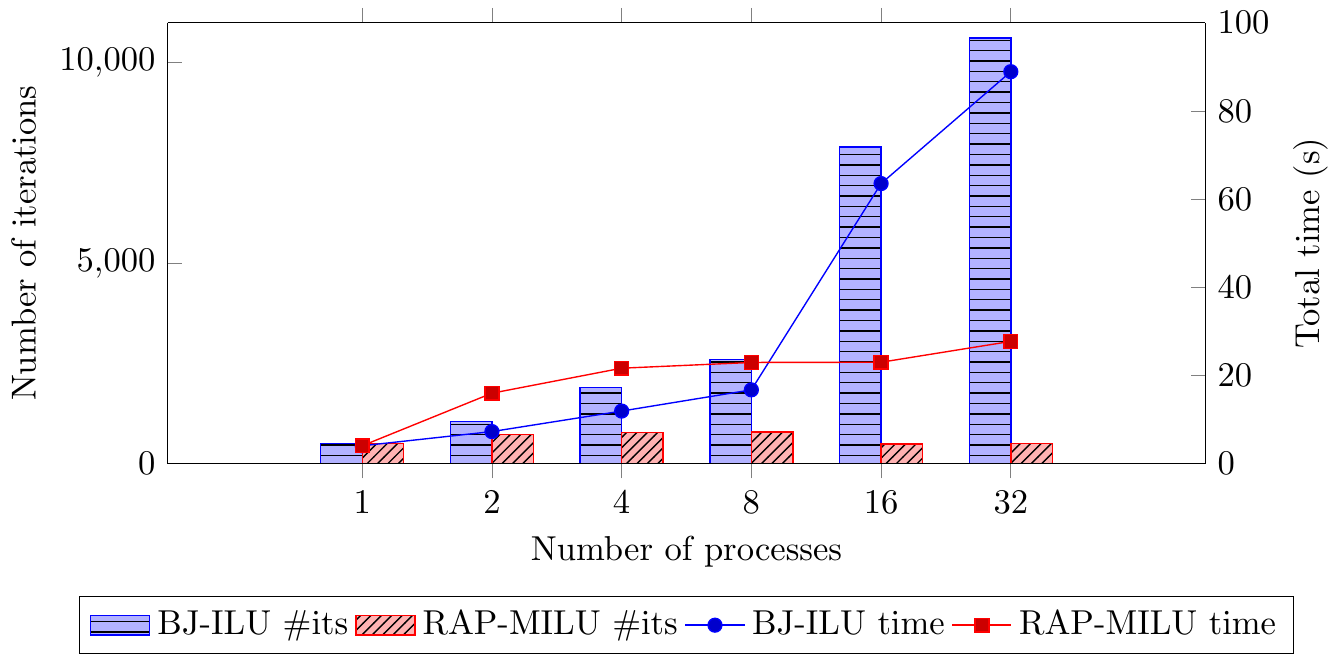} \\[0.5em]
        \def\arraystretch{1.1}
        {\fontsize{10}{10}\selectfont
          \setlength{\tabcolsep}{3pt}
          \begin{tabular}{rlrrr}
          \bf Np & \bf Preconditioner &\bf \#its & \bf setup & \bf solve  \\ \hline 
          \multirow{2}{*}{1}  & BJ-ILU & 507 & 0.42 & 2.65 \\
          & RAP-MILU & 507 & 0.42 & 2.65 \\ \hline
          \multirow{2}{*}{2}  & BJ-ILU & 1048 & 0.44 & 5.85 \\
          & RAP-MILU & 733 & 0.48 & 15.51\\ \hline
          \multirow{2}{*}{4}  & BJ-ILU & 1898 & 0.44 & 11.50 \\
          & RAP-MILU & 772 & 0.49 & 21.17 \\ \hline
          \multirow{2}{*}{8}  & BJ-ILU & 2598 & 0.45 & 16.30 \\
          & RAP-MILU & 793 & 0.51 & 22.45 \\ \hline
          \multirow{2}{*}{16} & BJ-ILU & 7902 & 0.45 & 63.07 \\
          & RAP-MILU & 489 & 0.51 & 22.50\\ \hline
          \multirow{2}{*}{32} & BJ-ILU & 10622 & 0.46 & 88.43 \\
          & RAP-MILU & 497 & 0.52 & 27.22\\ \hline
          \end{tabular}
          }
        \caption{2D Poisson problem}
    \end{subfigure}
    \begin{subfigure}[t]{.4\textwidth}
        \centering
        \includegraphics[width=\textwidth]{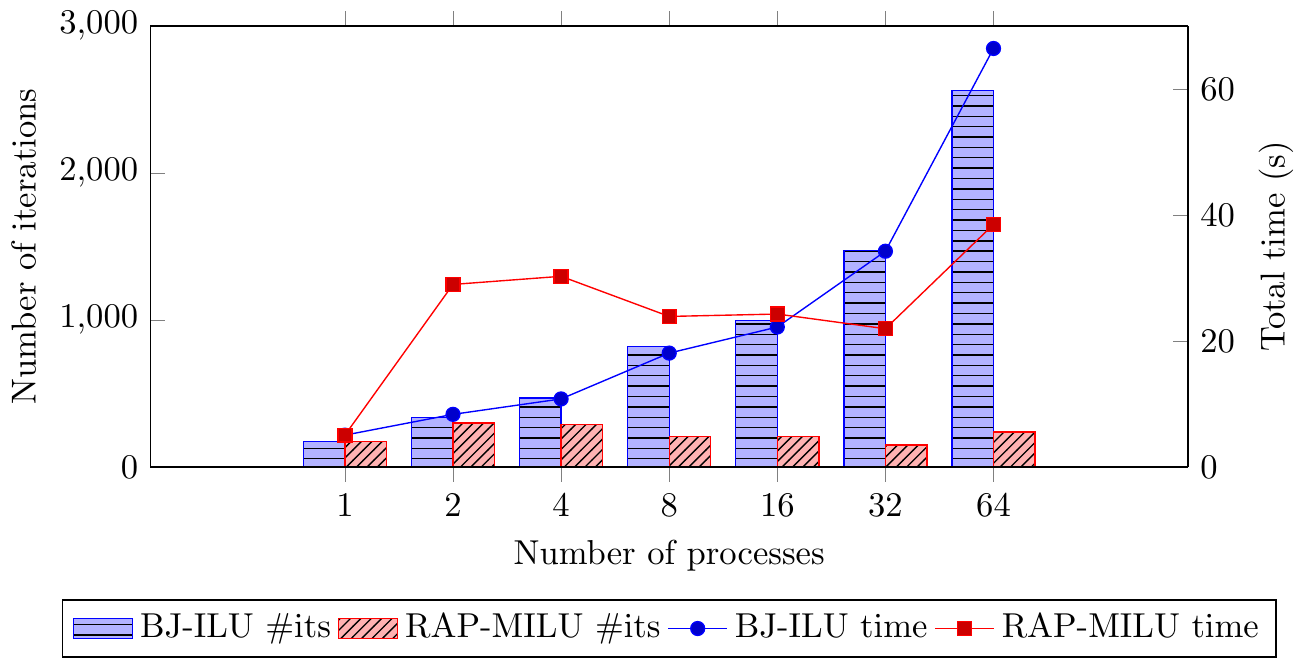} \\[0.5em]
        \def\arraystretch{1.1}
        {\fontsize{10}{10}\selectfont
          \setlength{\tabcolsep}{3pt}
          \begin{tabular}{rlrrr}
          \bf Np & \bf Preconditioner &\bf \#its & \bf setup & \bf solve  \\ \hline 
          \multirow{2}{*}{1}  & BJ-ILU & 175 & 1.86 & 3.21 \\
          & RAP-MILU & 175 & 1.86 & 3.21 \\ \hline
          \multirow{2}{*}{2}  & BJ-ILU & 336 & 1.89 & 6.48 \\
          & RAP-MILU & 300 & 5.40 & 23.59 \\ \hline
          \multirow{2}{*}{4}  & BJ-ILU & 470 & 1.67 & 9.15 \\
          & RAP-MILU & 290 & 5.20 & 25.08 \\ \hline
          \multirow{2}{*}{8}  & BJ-ILU & 819 & 1.43 & 16.66 \\
          & RAP-MILU & 208 & 4.38 & 19.53 \\ \hline
          \multirow{2}{*}{16} & BJ-ILU & 999 & 1.44 & 20.79 \\
          & RAP-MILU & 208 & 4.40 & 19.89 \\ \hline
          \multirow{2}{*}{32} & BJ-ILU & 1474 & 1.27 & 32.99 \\
          & RAP-MILU & 149 & 4.00 & 17.96 \\ \hline
          \multirow{2}{*}{64} & BJ-ILU & 2561 & 1.21 & 65.23 \\
          & RAP-MILU & 239 & 3.93 & 34.58 \\ \hline
          \end{tabular}
          }
        \caption{3D Poisson problem}
    \end{subfigure}
    \caption{\label{fig:test_weak}
    Weak scalability study of the block Jacobi  with ILU(0) and the two-level ILU(0)/MILU(0) preconditioners for 2-D/3-D Poisson problems along with
    FGMRES(50) with up to 64 processes on $16$ nodes of Ray.
    The problem size is $256^2$ dofs per process for the 2D problem and 
    $128^3$ for the 3D problem.}
\end{figure}

According to the results in  Figure~\ref{fig:test_weak}, the RAP-MILU preconditioner has better weak scaling results. The number of iterations is more stable as the number of  processes increases. In both cases, we observe an increase in the time to solution when the number of MPI processes is increased from 1 to 2. For the RAP-MILU preconditioner, this is primarily because RAP-MILU requires one classical ILU factorization and one modified ILU factorization, which increases the setup time. In addition, the solve time is also increased due to the additional MATVEC with $S$. For the BJ-ILU strategy, this is due to the lack of off-processor information in the factorization, which impacts its convergence properties. Thus, as we increase the number of MPI processes, we see a significant deterioration in the performance of BJ-ILU.   

{
Next, we evaluate the performance of the ILU solvers as preconditioners for an application in multiphase flow in porous media. The problem is the solution of a linear system of equations arising from the simulation of a two-phase, two-component compositional flow model of CO$_2$ injection. Details of the formulation and governing equations may be found in \cite{mgr-2021}. The resulting linear system captures the strong coupling between elliptic and hyperbolic dynamics associated with the unknowns. In addition, coupling between PDEs and algebraic constraint equations that govern well models at the injection sites leads to a complex linear system. These algebraic constraints also lead to equations with zero diagonal entries. This, coupled with the complex physical dynamics of the problem, leads to a linear system that is challenging to solve. Standard AMG is known to be ineffective for this problem, and 
strategies such as constrained pressure residual (CPR) \cite{Wallis83,Wallis85} and other block multilevel strategies have been proposed \cite{Wang17,mgr-2021} as efficient solution methods for this problem. CPR employs a two-stage process: a solve on a reduced system corresponding to the elliptic equations, followed by an approximate solve on the coupled global system. Both stages are essential for the convergence of CPR. The global system solve is typically done with ILU (BJ-ILU), while the reduced system solve is performed with a scalable elliptic PDE solver like AMG. To mitigate the cost of solving the global system, typically only a few iterations of the global solver is applied.

In the results that follow, we demonstrate the performance of BJ-ILU and SCHUR-ILU as a standalone solver, and as a preconditioner for an FGMRES solver. 
 The underlying linear system has a size of $614,400$ and we fix the restart dimension for FGMRES to 100. 
We note that our implementation of the ILU factorization perturbs small entries on the diagonal to avoid breakdown of the factorization. Thus, the factorization is still successful even if the linear system matrix has zero diagonals.
Figure~\ref{fig:comp_flow_res} shows the convergence profiles for the first 100 iterations. The results indicate that the SCHUR-ILU strategy yields a better convergence rate than BJ-ILU, and thus would be effective as an approximate solver for the global system of the coupled problem. In the table in Figure~\ref{fig:comp_flow_res}, we present results on a strong scaling evaluation of the same problem with up to 64 MPI processes. Here, we set the convergence tolerance for FGMRES to $10^{-5}$, and use ILU(2) factorization. The results indicate that the SCHUR-ILU preconditioner is less sensitive to the change in the number of subdomains, and exhibits better convergence when the  number of MPI processes is increased.



\begin{figure}[htb] 
\begin{center} 
\begin{subfigure}[t]{.3\textwidth}
        \centering
        \includegraphics[width=0.9\textwidth]{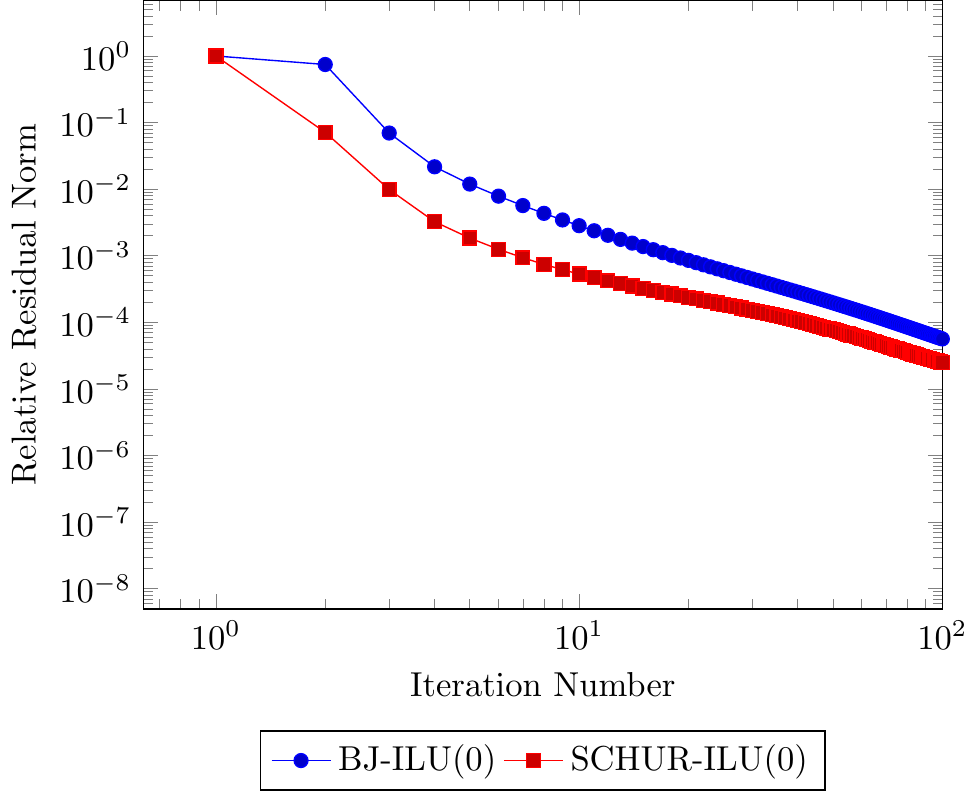}\hspace{0em}
        \caption{ILU(0) without FGMRES. Np=32.}
\end{subfigure}
\begin{subfigure}[t]{.3\textwidth}
        \centering
        \includegraphics[width=0.9\textwidth]{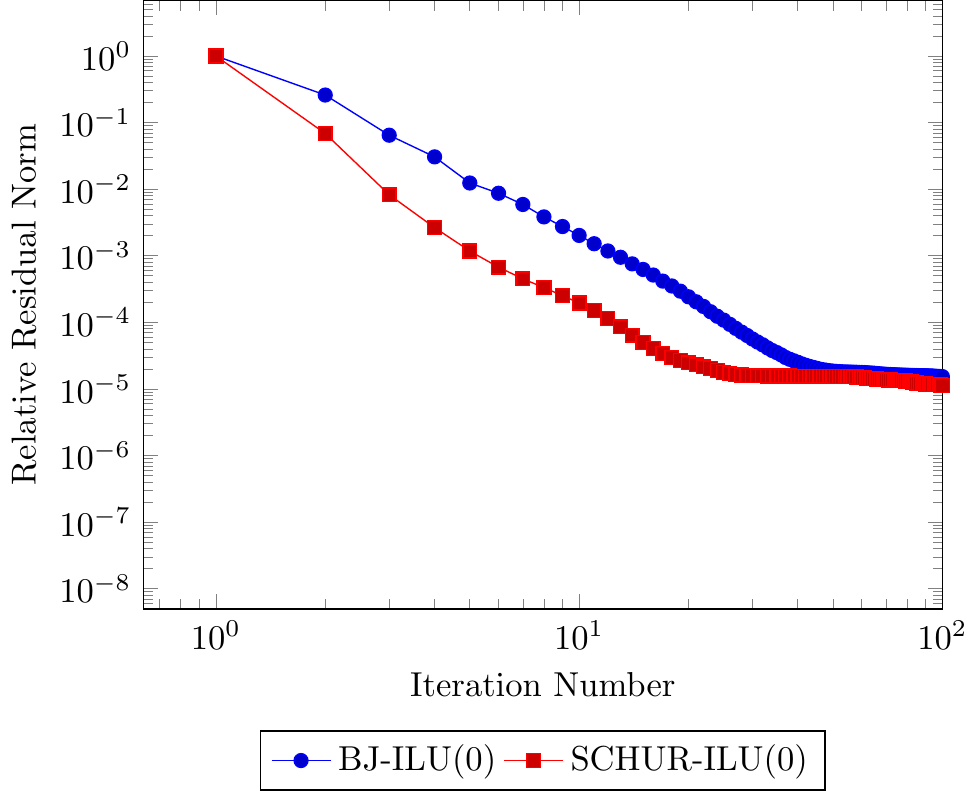}\hspace{0em}
        \caption{ILU(0) with FGMRES(100). Np=32.}
\end{subfigure}
\begin{subfigure}[t]{.3\textwidth}
        \centering
        \includegraphics[width=0.9\textwidth]{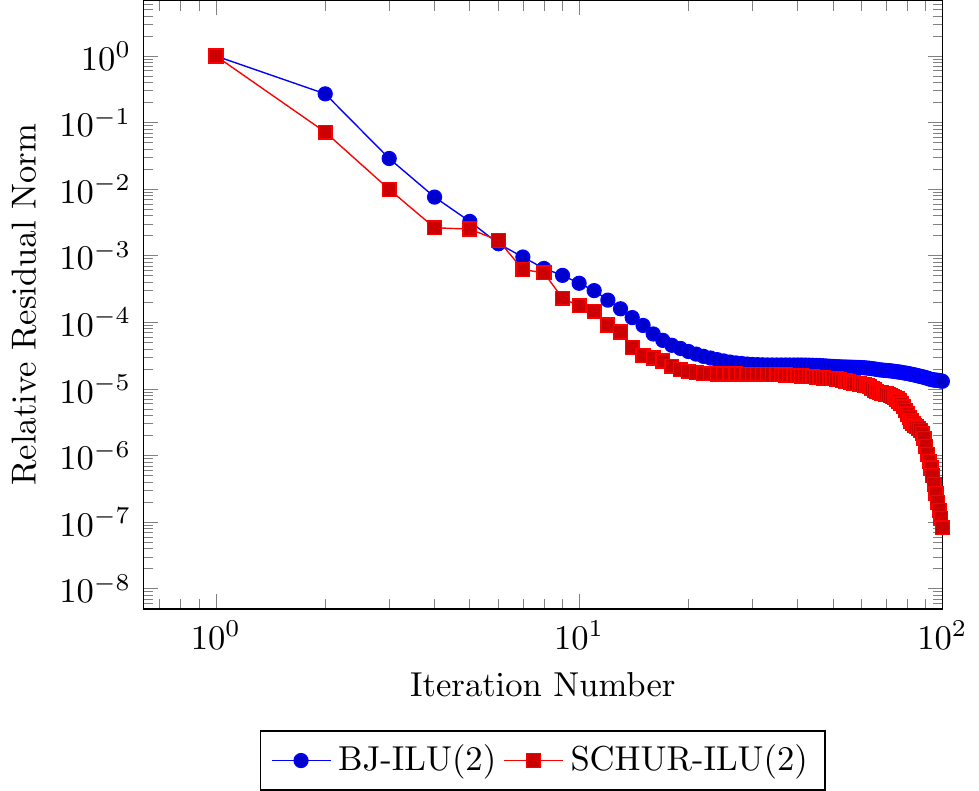} 
        \caption{ILU(2) with FGMRES(100). Np=32.}
\end{subfigure}
\begin{subtable}[tbhp]{1.0\textwidth}
    \centering
    \vspace{0.5em}
    \begin{tabular}{rrrr||rrr} 
   \bf \multirow{2}{*}{Np} & \multicolumn{3}{c||}{SCHUR-ILU} & \multicolumn{3}{c}{BJ-ILU} \\
   ~ & \bf \#its & \bf setup & \bf solve & \bf \#its & \bf setup & \bf solve \\ \hline 
   1 & 89 & 8.75 & 23.49 & 89 & 8.75 & 23.49 \\ 
   2 & 80 & 4.64 & 10.59 & 82 & 4.33 & 10.47 \\ 
   4 & 57 & 2.36 & 4.13 & 92 & 2.10 & 5.86 \\ 
   8 & 68 & 1.19 & 2.50 & 96 & 1.05 & 3.09 \\ 
   16 & 74 & 0.65 & 1.55 & 97 & 0.52 & 1.61 \\ 
   32 & 85 & 0.33 & 0.98 & 129 & 0.24 & 1.04\\ 
   64 & 94 & 0.17 & 0.64 & 222 & 0.11 & 0.88 \\ \hline
  \end{tabular}
  \caption{Strong scaling tests of BJ-ILU and SCHUR-ILU with ILU(2) as preconditioners for FGMRES(100). Up to 64 MPI processes are used.}
  \end{subtable}
\end{center} 
\caption{Relative residual norm, iteration counts and timings of the BJ-ILU and the two-level additive SCHUR-ILU for solving the compositional flow problem.}
\label{fig:comp_flow_res}
\end{figure}

}

\subsubsection{ILU as smoothers for AMG}

{ In the following experiments in this section, we use BoomerAMG in hypre to evaluate the performance of our two-level ILU strategy as smoothers for AMG. 
The test problem is a Laplace problem modelled on a crooked pipe domain. We employ a finite element discretization using the MFEM package \cite{mfem, mfem-web}. 
A rendering of the resulting unstructured mesh, using GLVis \cite{glvis-tool}, is shown in Figure~\ref{fig:mesh_crooked_pipe}. As shown in the figure, the mesh is inhomogeneous, which makes the problem challenging to solve. 

\begin{figure}[htb] 
\begin{center} 
\includegraphics[width=0.18\textwidth]{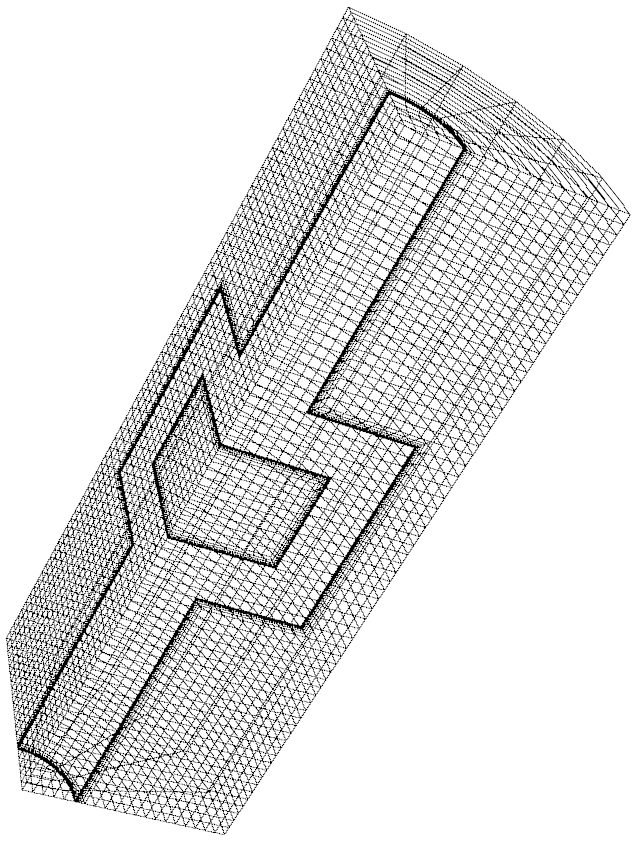}\hspace{6em}
\includegraphics[width=0.25\textwidth]{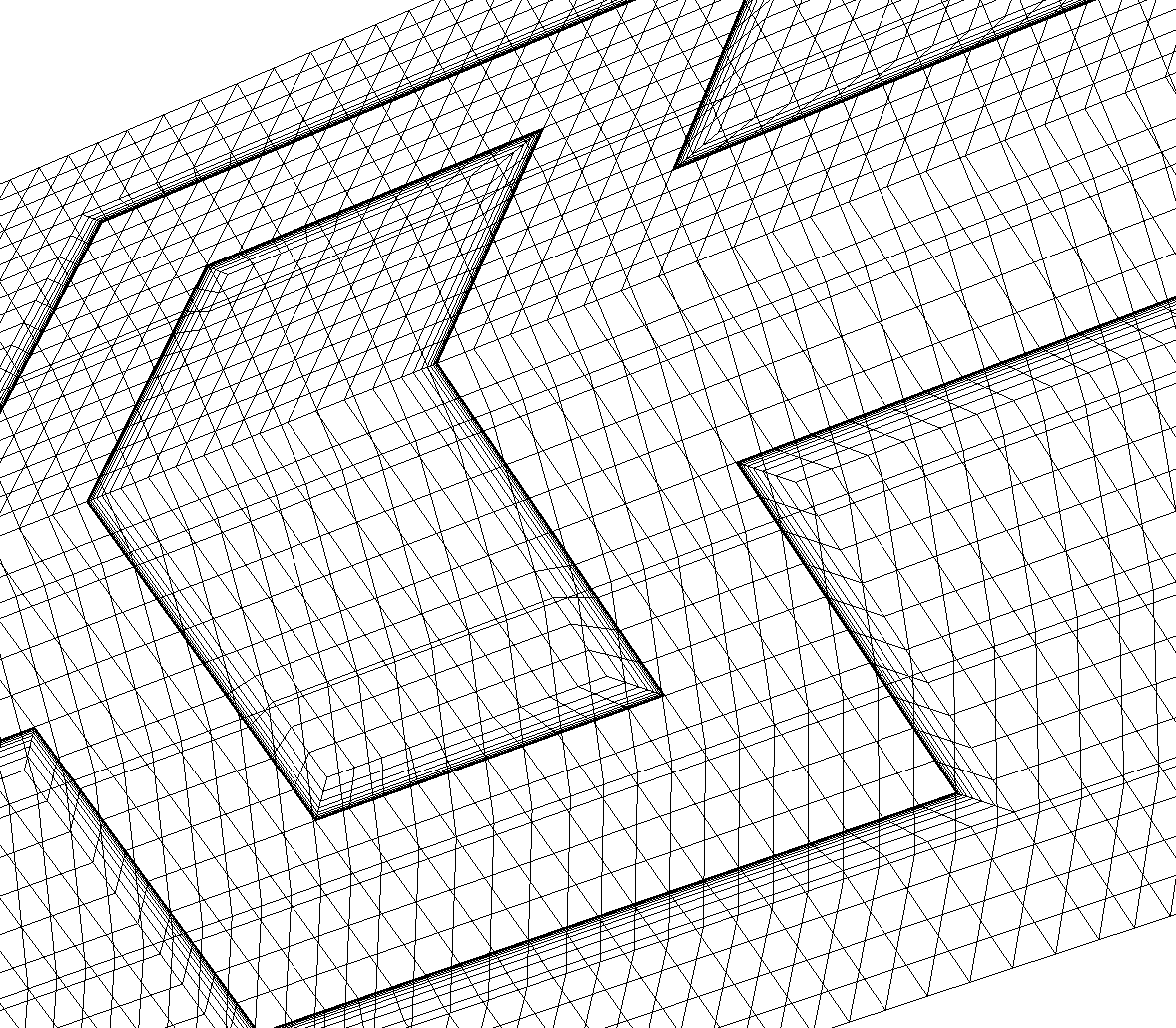}
\end{center} 
\caption{The mesh of the crooked pipe (left) and the zoom in of the center part (right).
\label{fig:mesh_crooked_pipe}}
\end{figure}

In the first set of experiments, we compare the $l_1$ Jacobi smoother with the two-level ILU smoothers.
The two-level ILU methods are used as the smoother for the finest level of AMG.
At the coarser levels, we use the $l_1$ Jacobi smoother for all cases. The problem is discretized using a first order finite element discretization and the mesh is uniformly refined to get the desired problem size.
We perform the tests on up to 32 MPI processes (without GPUs), using 2 nodes on the cluster Ray.
For these tests, we evaluate the performance of both BJ-ILU and the two-level SCHUR-ILU using ILU(1).
It is also worth mentioning that the standard BJ-ILU fails when the number of MPI processes is large. 
As a result, we use a variant of the BJ-ILU to improve the convergence rate.
Instead of ignoring all off-diagonal block entries as in standard BJ methods, 
we use an $l_1$-BJ variant where
we add the absolute values of the entries in off-diagonal blocks to the diagonal,
that is for \eqref{eq:locsys}, we have
\begin{equation}
\left(A_{i}\right)_{k,k}=\sum_{i\neq j}{\sum_l{\left|\left(E_{{i,j}}\right)_{k,l}\right|}} \ ,
\end{equation}
and then perform ILU factorization on $A_i$.
The $l_1$-BJ approach can guarantee a convergent smoother for AMG when $A$ is SPD and 
assuming the ILU factorization is accurate enough.
For more discussions and analysis of $l_1$ smoothers, see, e.g., \cite{smoothers}.

We run our tests with two different coarsening algorithms:
the parallel modified independent set  (PMIS) coarsening, which generates the same coarse grid with different numbers of MPI processes; and the 
hybrid modified independent set (HMIS) coarsening that adds one extra step after PMIS to improve mesh quality, and for which
however, the coarse grid can be very different for different number of MPI processes.
According to the results in Table~\ref{tab:test_amg_bj_vs_gmres}, the SCHUR-ILU smoother yields the best strong scaling results.
For both coarsening algorithms, SCHUR-ILU is less sensitive to the change of the number of MPI processes.

\begin{table}[tbhp]
\centering
    \caption{Iteration counts and timings of the $l_1$-block-Jacobi smoother with ILU(1) and the two-level additive ILU(1) smoother along with AMG for solving the crooked pipe problem. The number of unknowns is 966,609.}
    \label{tab:test_amg_bj_vs_gmres}
    \begin{subtable}[tbhp]{1.0\textwidth}
    \centering
    \begin{tabular}{rrrr||rrr||rrr} 
   \bf \multirow{2}{*}{Np} & \multicolumn{3}{c||}{$l_1$-Jacobi} & \multicolumn{3}{c||}{$l_1$-BJ-ILU} & \multicolumn{3}{c}{SCHUR-ILU} \\
   ~ & \bf \#its & \bf setup & \bf solve & \bf \#its & \bf setup & \bf solve & \bf \#its & \bf setup & \bf solve \\ \hline 
   1 & 96 & 8.38 & 38.75  & 41 & 23.29 & 37.36 & 41 & 23.29 & 37.36 \\ 
   2 & 141 & 5.26 & 29.69  & 85 & 12.75 & 39.35 & 43 & 13.55 & 20.39 \\ 
   4 & 155 & 3.42 & 16.88  & 80 & 7.16 & 18.79 & 48 & 7.55 & 12.17 \\ 
   8 & 100 & 1.89 & 5.86  & 54 & 3.72 & 6.57 & 46 & 3.93 & 6.26 \\ 
   16 & 146 & 1.15 & 4.63  & 92 & 2.09 & 5.83 & 45 & 2.20 & 3.35 \\ 
   32 & 174 & 0.65 & 2.97  & 106 & 1.12 & 3.45 & 45 & 1.19 & 1.86\\ \hline
  \end{tabular}
  \caption{AMG with HMIS coarsening}
  \end{subtable}
  \begin{subtable}[tbhp]{1.0\textwidth}
    \centering
  \begin{tabular}{rrrr||rrr||rrr} 
   \bf \multirow{2}{*}{Np} & \multicolumn{3}{c||}{$l_1$-Jacobi} & \multicolumn{3}{c||}{$l_1$-BJ-ILU} & \multicolumn{3}{c}{SCHUR-ILU} \\
   ~ & \bf \#its & \bf setup & \bf solve & \bf \#its & \bf setup & \bf solve & \bf \#its & \bf setup & \bf solve \\ \hline 
   1 & 284 & 7.68 & 100.64  & 75 & 22.43 & 64.76 & 75 & 22.43 & 64.76 \\ 
   2 & 280 & 4.38 & 51.63  & 102 & 11.82 & 44.90 & 74 & 12.94 & 33.08 \\ 
   4 & 291 & 2.65 & 27.91  & 112 & 6.35 & 24.99 & 72 & 6.76 & 17.25 \\ 
   8 & 295 & 1.53 & 14.86  & 107 & 3.35 & 12.16 & 74 & 3.56 & 9.47 \\ 
   16 & 293 & 0.92 & 8.13  & 102 & 1.85 & 6.13 & 75 & 1.97 & 5.31 \\ 
   32 & 283 & 0.52 & 4.32  & 114 & 0.98 & 3.48 & 73 & 1.05 & 2.87\\ \hline
  \end{tabular}
  \caption{AMG with PMIS coarsening}
  \end{subtable}
\end{table}

In the next set of experiments, we evaluate the performance of SCHUR-ILU with different mesh sizes and different orders for the finite element discretization.
We use the same settings as the previous set of experiments, except that the number of MPI processes is fixed at 32.
As the results in Table~\ref{tab:test_amg} indicate, the SCHUR-ILU smoother typically works better than $l_1$ Jacobi smoother.
By replacing the smoother on the first level with SCHUR-ILU, we observe a reduction in the number of AMG cycles, which generally leads to a much faster solve phase.
The SCHUR-ILU smoother requires some extra time in the setup phase, but the total time is in general smaller.
Thus, it is easy to see that performance could be even better in a situation where multiple right-hand-sides are solved with the same linear system.

\begin{table}[tbhp]
\centering
    \caption{Iteration counts and timings of the $l_1$-Jacobi smoother
    and the additive two-level ILU smoother with ILU(1) for solving the crooked pipe problem with different mesh sizes and different orders of finite elements. The runs used 32 MPI processes on 2 nodes.}
    \label{tab:test_amg}
 \begin{tabular}{crrrrr} 
  \bf Order & \bf \#Unknowns  & \bf Smoother & \bf \#its & \bf setup & \bf solve \\ \hline 
  \multirow{2}{*}{1} & \multirow{2}{*}{126,805} & $l_1$-Jacobi & 82 & 0.09 & 0.20 \\ 
    & & SCHUR-ILU & 48 & 0.15 & 0.30 \\ \hline 
   \multirow{2}{*}{1} & \multirow{2}{*}{966,609} & $l_1$-Jacobi & 174 & 0.64 & 2.97 \\ 
    & & SCHUR-ILU & 45 & 1.19 & 1.86 \\ \hline 
    \multirow{2}{*}{1} &\multirow{2}{*}{7,544,257} &  $l_1$-Jacobi & 212 & 6.91 & 28.48 \\ 
    & & SCHUR-ILU & 58 & 11.37 & 18.09 \\ \hline 
    \multirow{2}{*}{2} &\multirow{2}{*}{126,805} &  $l_1$-Jacobi & 212 & 0.10 & 0.80 \\ 
    & & SCHUR-ILU & 30 & 0.30 & 0.45 \\ \hline 
    \multirow{2}{*}{2} &\multirow{2}{*}{966,609} &  $l_1$-Jacobi & 464 & 0.72 & 13.00 \\ 
    & & SCHUR-ILU & 47 & 2.28 & 4.58 \\ \hline 
    \multirow{2}{*}{3} &\multirow{2}{*}{414,472} &  $l_1$-Jacobi & 285 & 0.53 & 6.08 \\ 
    & & SCHUR-ILU & 27 & 2.45 & 2.58 \\ \hline 
    \multirow{2}{*}{3} &\multirow{2}{*}{3,209,173} &  $l_1$-Jacobi & 694 & 3.93 & 121.43 \\ 
    & & SCHUR-ILU & 34 & 18.57 & 19.50 \\ \hline
  \end{tabular}
\end{table}

These results demonstrate the usefulness of ILU methods as smoothers for AMG to solve challenging problems where simple smoothers lead to a slow convergence. 
}

\subsection{GPU Speedup} 

{
In the following experiments, we study the GPU speedup of our implementation by comparing the solve time of the same problem with and without GPU acceleration.
The speedup of the triangular solves for the Nvidia cuSPARSE library has been studied in detail \cite{naumov2011parallel}.
Here, we evaluate the GPU speedup of the two-level ILU methods as preconditioners for FGMRES, and the effect of the GPU version of the Schur complement strategy on the convergence rate.

We compare our GPU and CPU variants with all the available computing power on a single node. 
Since our algorithms are DD-based, the number of subdomains could influence the convergence. 
we use the same number of subdomains for the CPU runs, and the GPU runs. 
In addition, for the CPU tests, we enable OpenMP threading to use all the CPU cores on a single node. 
Thus, for the GPU tests on Ray, we use 4 MPI processes, each of which is bound to a GPU. 
For the corresponding CPU tests, we use 4 MPI processes, each with 5 OpenMP threads. 
The tests on Lassen have a similar setup, except that the number of OpenMP threads for each MPI process is 11 for the CPU runs.
To bind MPI processors to physical cores, we use \texttt{mpibind} \cite{10.1145/3132402.3132415}, which, on a single node, binds by socket first.

We report the performance of the CPU variant and GPU variant of BJ-ILU and two-level ILU methods and compute the speedup of all these options. 
We perform our evaluation on two sets of problems. 
In the first example, we consider a 3D discretized Laplacian problem on a $128^3$ domain. 
We use ILU(0) and FGMRES(50) and report the speedup of the setup and the solve phase. 
Table \ref{tab:lap2} shows the results obtained on Ray (P100 GPU) and Lassen (V100 GPU). 
As the results indicate, we can have a total speedup of a factor of 3.38 for BJ-ILU, 2.52 for SCHUR-ILU, and 1.34 for RAP-MILU on Ray with P100 GPUs. 
With the V100 GPUs on Lassen, these numbers are 3.0, 1.2, and 1.74, respectively.

\begin{table}[tbhp]
\centering
\caption{Iteration counts and timings (in seconds) of the CPU and the GPU implementations of the block-Jacobi preconditioner with ILU(0) and the two-level ILU(0) preconditioners along with FGMRES(50) for solving 3D Poisson problem 
of size $128^3$ with 7pt Laplacians}
\label{tab:lap2}
\begin{subtable}[tbhp]{0.4\textwidth}
\fontsize{10}{12}\selectfont
\def\arraystretch{1.1}
\setlength{\tabcolsep}{3.5pt}
  \begin{tabular}{lrrrr} 
   \bf Preconditioner & \bf Device & \bf \#its & \bf setup & \bf solve \\  \hline
    \multirow{2}{*}{BJ-ILU} &CPU & 229 & 0.17 & 8.35 \\ 
     &GPU & 229 & 0.64 & 1.88 \\ \hline
    \multirow{2}{*}{SCHUR-ILU} &CPU & 175 & 0.20 & 9.35 \\ 
     &GPU & 175 & 0.65 & 3.14 \\ \hline
    \multirow{2}{*}{RAP-MILU} &CPU & 154 & 0.20 & 9.11 \\ 
     &GPU & 154 & 1.23 & 5.70 \\ \hline
  \end{tabular}
\caption{P100 GPU (Ray)}
\end{subtable}
\begin{subtable}[tbhp]{0.4\textwidth}
\fontsize{10}{12}\selectfont
\def\arraystretch{1.1}
\setlength{\tabcolsep}{3.5pt}
  \begin{tabular}{lrrrr} 
   \bf Preconditioner & \bf Device & \bf \#its & \bf setup & \bf solve \\ \hline 
    \multirow{2}{*}{BJ-ILU} &CPU & 229 & 0.15 & 6.07 \\ 
     &GPU & 229 & 0.86 & 2.04 \\ \hline
    \multirow{2}{*}{SCHUR-ILU} &CPU & 175 & 0.18 & 6.02 \\ 
     &GPU & 175 & 0.7 & 5.28 \\ \hline
     \multirow{2}{*}{RAP-MILU} &CPU & 154 & 0.27 & 17.96 \\ 
     &GPU & 154 & 1.78 & 10.34 \\ \hline
  \end{tabular}
\caption{V100 GPU (Lassen)}
\end{subtable}
\end{table}

{
For the second example, we consider the compositional multiphase flow problem described earlier. The problem size is approximately 4.5 million. Recall that the problem formulation leads to the presence of zeros on the diagonal of the resulting coefficient matrix. 
This makes it infeasible to use the ILU(0) factorization in the cuSPARSE library. 
As a result, we perform the factorization on the CPU and transfer the resulting factors to the GPU for the solve phase. We use ILUT and ILU(1) factorizations for this evaluation. 
We note that the fill-factor of ILU(1) is roughly 1.8 and we choose a threshold such that the fill-factor for ILUT is also close to 1.8.
We run 100 iterations of FGMRES(50) without restart and report the speedup the solve phase only.
To allow for a fair comparison, we adjust the solver options for the two-level ILU preconditioner so that the relative residual after 100 iterations is similar between the CPU and GPU.
Table~\ref{tab:comp_gpu} show the results obtained on Ray (P100 GPU) and Lassen (V100 GPU). 
The results indicate speedups of 2.62 for BJ-ILU with ILUT, 2.48 for BJ-ILU with ILU(1), 2.52 for Schur-ILU with ILUT, and 2.64 for Schur-ILU with ILU(1) on Ray with P100 GPUs.
With the V100 GPUs on Lassen, these numbers are 1.78, 2.09, 2.12, and 2.14, respectively.
}

Notice that on Lassen, the GPU acceleration of the two-level ILU solve is much slower than that on Ray for the 3D discretized Laplacian problem. 
We found that this increase in solution time was caused by the performance of the 
sparse triangular solves associated with the local Schur complements. On the V100 GPUs, the sparse triangular solve operation is considerably slower than that on the P100 GPUs, and even slower than that on the CPU. We speculate that this is because the triangular factors are too small for the V100 GPUs to get a good speedup. We are currently investigating this issue further.

\begin{table}[tbhp]
\centering
\caption{Timings (in seconds) of the solve phase of the CPU and the GPU implementations of the block-Jacobi preconditioner with ILUT/ILU(1) and the two-level ILUT/ILU(1) preconditioners along with 100 steps of FGMRES for solving a compositional flow problem.}
\label{tab:comp_gpu}
\begin{subtable}[tbhp]{0.4\textwidth}
\fontsize{10}{12}\selectfont
\def\arraystretch{1.1}
\setlength{\tabcolsep}{3.5pt}
  \begin{tabular}{lrrr} 
   \bf Preconditioner & \bf Device & \bf ILUT & \bf ILU(1) \\  \hline
    \multirow{2}{*}{BJ-ILU}    &CPU & 10.83 & 11.96 \\ 
                               &GPU & 4.13 & 4.83 \\ \hline
    \multirow{2}{*}{SCHUR-ILU} &CPU & 17.03 & 15.67 \\ 
                               &GPU & 5.30 & 5.93 \\ \hline
  \end{tabular}
\caption{P100 GPU (Ray)}
\end{subtable}
\begin{subtable}[tbhp]{0.4\textwidth}
\fontsize{10}{12}\selectfont
\def\arraystretch{1.1}
\setlength{\tabcolsep}{3.5pt}
  \begin{tabular}{lrrrr}
   \bf Preconditioner & \bf Device & \bf  ILUT & \bf ILU(1) \\ \hline 
    \multirow{2}{*}{BJ-ILU}    &CPU & 8.87 & 10.05 \\
                               &GPU & 4.98 & 4.80 \\ \hline
    \multirow{2}{*}{SCHUR-ILU} &CPU & 13.04 & 14.35 \\
                               &GPU & 6.15 & 6.69 \\ \hline
  \end{tabular}
\caption{V100 GPU (Lassen)}
\end{subtable}
\end{table}

}

\subsection{Two-level ILU methods with ILU(k)} 

In the final set of experiments, we compare the performance of our proposed two-level ILU implementation to the two-level graph-based parallel ILU implementation in Euclid \cite{hysom1999efficient}. Euclid is also available through hypre, which makes is convenient to perform a fair evaluation of both strategies. The evaluation is performed on CPU only since Euclid is not GPU-enabled, and we focus on ILU($k$), which is the parallel ILU variant supported by Euclid. The level of fill is chosen to be 2.

The model problem is a 3D convection diffusion problem defined as:
\begin{equation} \label{eq:difconv}
{-\Delta u + b \cdot \nabla u = f \quad \mbox{in} \quad  [0,1]^3}
\end{equation}
and is discretized via finite differences with a 7-pt stencil. 
We perform strong scaling tests with MPI processes ranging from 1 to 16 on a problem of size $128^3$, and 4 to 64 on a problem of size $256^3$.
The purpose of this experiment is to demonstrate the efficiency of the proposed ILU strategies compared to Euclid-ILU, implemented within the hypre linear solver library. To facilitate this comparison, we use the parallel ILU techniques as preconditioners for FGMRES(100) in order to obtain convergence for Euclid-ILU on these problems.
The results are presented in Figure~\ref{fig:test_gmres_euclid}.

\begin{figure}[!tbhp]
\centering
    \begin{subfigure}[t]{.4\textwidth}
        \centering
        \includegraphics[width=\textwidth]{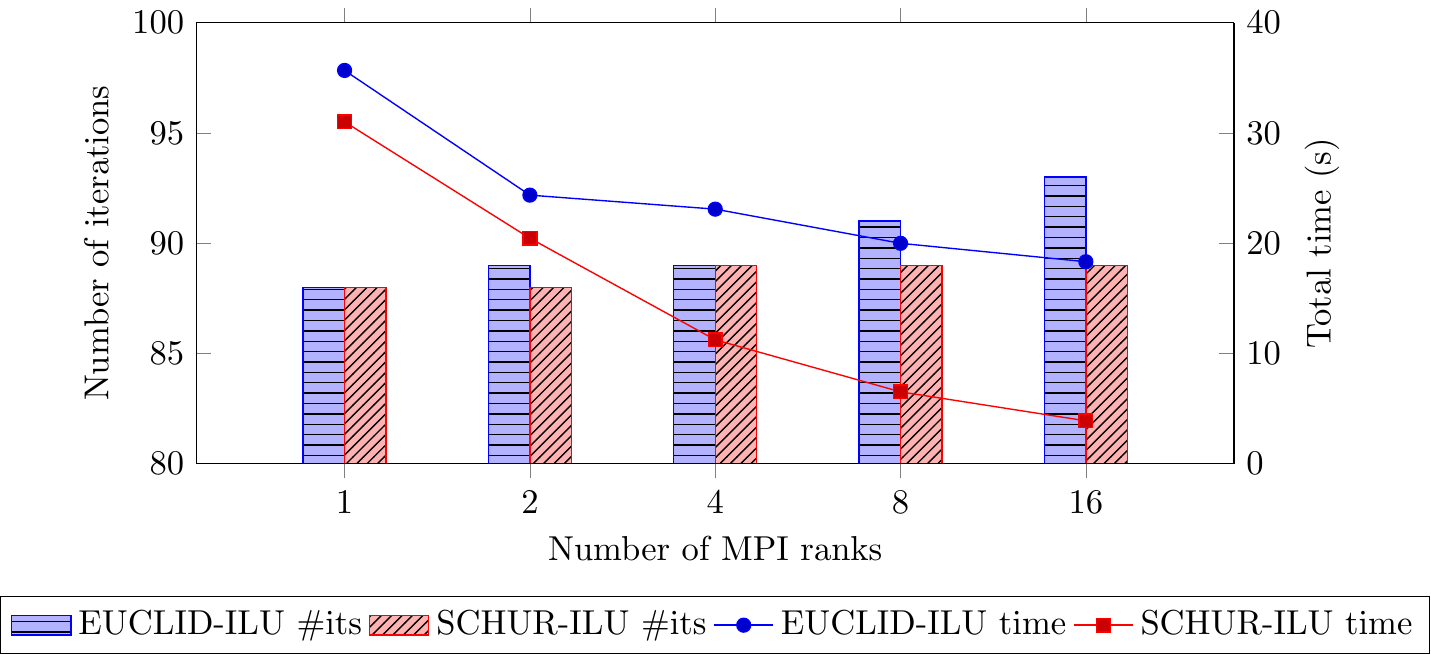} \\[0.5em]
        \setlength{\tabcolsep}{3pt}
      \begin{tabular}{rrrrr} 
  \bf Np & \bf Preconditioner & \bf \#its & \bf setup & \bf solve \\ \hline 
    \multirow{2}{*}{1} & EUCLID-ILU & 88 & 8.42 & 27.25 \\
    & SCHUR-ILU & 88 & 2.37 & 28.64 \\ \hline
    \multirow{2}{*}{2} & EUCLID-ILU & 89 & 10.14 & 14.22 \\
    & SCHUR-ILU & 88 & 1.15 & 19.30 \\ \hline
    \multirow{2}{*}{4} & EUCLID-ILU & 89 & 15.59 & 7.50 \\ 
    & SCHUR-ILU & 89 & 0.65 & 10.59 \\ \hline
    \multirow{2}{*}{8} & EUCLID-ILU & 91 & 15.60 & 4.39 \\ 
    & SCHUR-ILU & 89 & 0.33 & 6.19 \\ \hline
    \multirow{2}{*}{16} & EUCLID-ILU & 93 & 15.72 & 2.60 \\
    & SCHUR-ILU & 89 & 0.18 & 3.71 \\ \hline
  \end{tabular}
        \caption{3D convection diffusion problem of size $128^3$}
    \end{subfigure}
    \begin{subfigure}[t]{.4\textwidth}
        \centering
        \includegraphics[width=\textwidth]{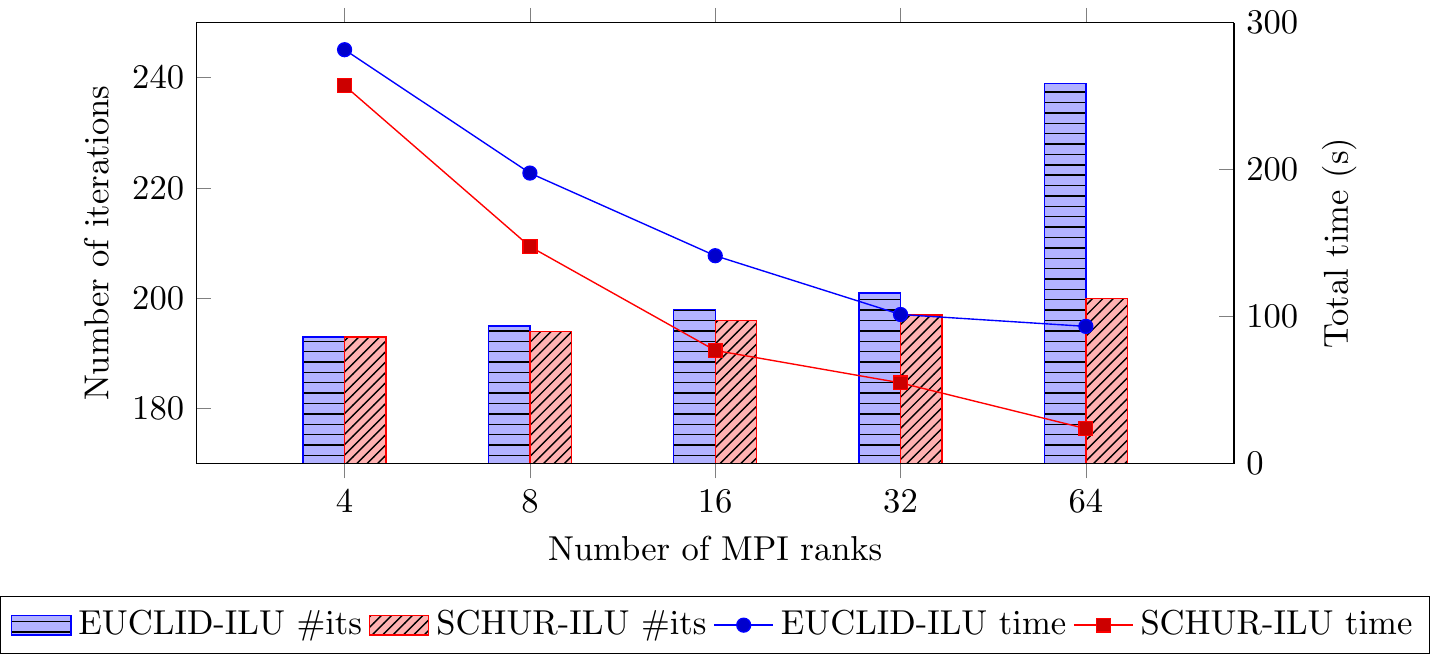} \\[0.5em]
                \setlength{\tabcolsep}{3pt}
          \begin{tabular}{rrrrr} 
          \bf Np & \bf Preconditioner & \bf \#its & \bf setup & \bf solve \\ \hline 
            \multirow{2}{*}{4} & EUCLID-ILU & 193 & 86.17 & 195.44 \\ 
            & SCHUR-ILU & 193 & 5.18 & 252.02 \\ \hline
            \multirow{2}{*}{8} & EUCLID-ILU & 195 & 83.86 & 113.83 \\ 
            & SCHUR-ILU & 194 & 2.83 & 144.9 \\ \hline
            \multirow{2}{*}{16} & EUCLID-ILU & 198 & 81.04 & 60.44 \\ 
            & SCHUR-ILU & 196 & 1.41 & 75.64 \\ \hline
            \multirow{2}{*}{32} & EUCLID-ILU & 201 & 77.35 & 24.15 \\ 
            & SCHUR-ILU & 197 & 0.79 & 54.21 \\ \hline
            \multirow{2}{*}{64} & EUCLID-ILU & 239 & 79.34 & 14.08 \\ 
            & SCHUR-ILU & 200 & 0.38 & 23.50 \\ \hline
          \end{tabular}
        \caption{3D convection diffusion problem of size $256^3$}
    \end{subfigure}
    \caption{Strong scalability study of the EUCLID-ILU preconditioner and SCHUR-ILU preconditioner for 3D convection diffusion problems using
    FGMRES(100).}
    \label{fig:test_gmres_euclid}
\end{figure}
The results indicate that both parallel ILU strategies have similar convergence behavior,
with Euclid exhibiting a slightly better solve time. However, the SCHUR-ILU strategy exhibits a superior strong scaling property in the setup costs, leading to a more efficient parallel ILU strategy overall.

Algebraic black-box solvers with multiple options, such as those presented in this paper have the benefit of being applicable to a wide class of problems. However, choosing the right options for a particular problem may require evaluating multiple options. 
Generally, it is appropriate to first consider cheaper strategies such as BJ-ILU(0) and RAP-MILU(0). 
Schur-ILU can be used if the convergence is not satisfactory.
Increasing the level of fill for ILU(k) or switching to ILUT with a small drop tolerance should typically further improve the convergence.
For most problems, we find that applying 3 iterations of GMRES, with ILU(0) as the preconditioner, on the coarse grid of the SCHUR-ILU method is enough to achieve good convergence. For more challenging problems, a more accurate ILU factorization and/ or a larger Krylov subspace dimension can be used to achieve better convergence. 
These tunable options enable users to control the cost vs. accuracy of the coarse grid solve, and thus, the performance of the preconditioner as a whole.

\section{Conclusion}\label{sec:conclusion}

Motivated by the reliability of ILU preconditioners, we proposed a DD-based two-level parallel ILU strategy designed for distributed memory systems and with GPU support. We presented a detailed analysis of the two-level strategy, highlighting additive and multiplicative variants for constructing the Schur complement problem defining the global coupling between subdomains. We also demonstrate the benefit of using modified ILU to improve the spectral properties of the global Schur complement system, leading to a more accurate preconditioner. We adapt our algorithms to GPUs by leveraging existing optimized kernels for performing ILU(0) factorizations and triangular solves. Our implemented algorithms are readily available in hypre.

Numerical results demonstrate the effectiveness and scalability of the two-level parallel ILU preconditioner variants compared to the standard block Jacobi strategy. 
This is particularly evident when the size of the global Schur complement is large, corresponding to when a large number of processors are used. Nonetheless, for smaller number of subdomains or problems with weak inter-domain coupling, block Jacobi can be quite competitive. One good strategy, therefore, is to use block Jacobi whenever possible and switch to the more accurate two-level strategy when necessary.  Our numerical results also indicate good GPU speedup. The amount of speedup typically depends on the matrix structure and problem size. For large problems with a good level structure in the triangular solve, GPU acceleration can yield significant performance gains compared to the CPU version. 

Future work will consider implementing CUDA kernels that support permutation arrays for efficient ILU factorizations and triangular solves.  We will also consider a multilevel ILU framework and its GPU version, to improve efficiency of the global Schur complement solve.

\bibliographystyle{plain}
\bibliography{papers,local}

\end{document}